\documentclass[10pt,a4paper]{article}

\usepackage[english]{babel}
\usepackage[utf8]{inputenc}

\usepackage{amsmath,amssymb,amsfonts,amsthm}
\usepackage{mathtools}
\usepackage{mathrsfs}
\usepackage{bbold}
\usepackage{latexsym}
\usepackage{cancel}

\usepackage{graphicx}
\usepackage{subcaption}
\usepackage{caption}
\usepackage{multicol}
\usepackage{fullpage}
\usepackage{floatrow}
\usepackage{multirow}

\usepackage{tikz}
\usepackage{tikz-cd}
\usetikzlibrary{arrows,decorations.pathmorphing,cd}

\usepackage{xcolor}
\usepackage{hyperref}
\hypersetup{
  colorlinks=true,
  linkcolor=blue,
  citecolor=blue,
  urlcolor=blue
}

\newtheorem{teorema}{Theorem}[section]
\newtheorem{definizione}[teorema]{Definition}

\newtheorem{osservazione}[teorema]{Remark}

\title{Invariance properties of Brownian motion via Lie's symmetries}
\date{}

\author{
Susanna Deh\`o\footnote{Dipartimento di Matematica, Universit\`a degli Studi di Milano, \texttt{susanna.deho@unimi.it}}
\and
Francesco C.~De Vecchi\footnote{Dipartimento di Matematica ``Felice Casorati'', Universit\`a degli Studi di Pavia, \texttt{francescocarlo.devecchi@unipv.it}}
\and
Paola Morando\footnote{DISAA, Universit\`a degli Studi di Milano, \texttt{paola.morando@unimi.it}}
\and
Stefania Ugolini\footnote{Dipartimento di Matematica, Universit\`a degli Studi di Milano, \texttt{stefania.ugolini@unimi.it}}
}

\begin{document}
\maketitle

\begin{abstract}
The invariance properties of Brownian motion are investigated and revisited within a recent Lie symmetry approach to stochastic differential equations. Some notable properties of the process can be recovered by a related integration by parts formula developed in the same research area.
\end{abstract}



\section*{Introduction}
The study of symmetries in differential equations, pioneered by Sophus Lie, constitutes a fundamental geometric approach to understand the invariance properties governing a dynamical system (\cite{Olver},\cite{olver2}). By identifying the groups of transformations that preserve the equation's structure, this framework provides systematic tools for order reduction, the construction of exact solutions, and the simplification of complex problems (\cite{gaetanonlinear}). Although symmetry analysis stands as a classical pillar for deterministic ordinary and partial differential equations (ODEs and PDEs), supported by extensive literature, its extension to Stochastic Differential Equations (SDE) represents a relatively recent field of research (\cite{deLara},\cite{ortega1998symmetry},\cite{ortega2013momentum},\cite{gaetalie},\cite{paper2015},\cite{paperriduzione1},\cite{paper}, \cite{romano},\cite{albeverio}, \cite{paper2020},\cite{paperriduzione},\cite{grandemostro},\cite{noether}). The reduction techniques that have proven useful in the deterministic setting can also be formulated in the stochastic case (\cite{lazaro2009reduction},\cite{paperriduzione},\cite{paperriduzione1}). However, properly extending the classical Lie symmetry approach to SDEs requires careful attention: the generality and richness of stochastic processes demand precise definitions and a clear distinction between different types of symmetries and transformations.

In this paper, we focus on Brownian SDEs and their invariance properties, emphasizing, in particular, Brownian motion itself, which, as a well-known and extensively studied case, provides a clear and illustrative example to highlight the full potential of our approach and its unifying perspective. We show that these invariance properties can be investigated both through the recent approach to Lie symmetries of SDEs (developed in Milan, \cite{paper2015}, \cite{paperriduzione1}, \cite{paper}, \cite{romano}, \cite{albeverio}, \cite{paper2020},\cite{paperriduzione}, \cite{grandemostro},\cite{noether}), and via a novel related integration by parts formula (\cite{paper2023},\cite{papernuovo}). \\
\\
The cornerstone of the theory of symmetries for SDEs is provided by stochastic transformations. While classical transformations of differential equations consist primarily of time changes $t \mapsto f(t)$ and space-time diffeomorphisms $(x,t) \mapsto \Phi(x,t)$, the framework of Brownian-driven SDEs admits at least two additional operations intimately linked to the noise term: random rotations of the driving Brownian motion and changes of measure via Girsanov's theorem (\cite{paper2015},\cite{albeverio}).

Let us analyze how stochastic transformations operate on a solution $(X,W)$ of a given SDE in a filtered probability space:
\begin{equation}\label{SDE intro}
    dX_t=\mu(X_t,t)dt+\sigma(X_t,t)dW_t.
\end{equation}
First, we observe that a stochastic transformation \( T\) acts simultaneously on the solution process and on the SDE itself, ensuring structural consistency: the transformed process remains a solution to the transformed equation, possibly within a transformed filtered probability space. Consequently, we denote by $P_T$ the action of the transformation $T$ on the solution process $(X,W)$, and by $E_T$ its action on the equation's coefficients $(\mu, \sigma)$. Mirroring the standard distinction between strong and weak solutions to SDEs, we distinguish between \textit{strong stochastic transformations}, which modify neither the driving Brownian motion nor the filtered probability space (fixed a priori), and \textit{weak stochastic transformations}. The latter, as it is well-known, may change the driving Brownian motion, the filtration, or the underlying probability measure. For weak transformations, admissibility is a crucial requirement: the transformed driving process must remain a Brownian motion with respect to the (possibly transformed) filtered probability space.

Let us start by considering a diffeomorphism $T=\Phi \in C^{2,1}$. By applying Itô's lemma to $\Phi(X_t,t)$, one immediately obtains:
\[
d\Phi(X_t,t) = L(\Phi)(\Phi^{-1}(\Phi(X_t,t)),t) dt + ((D\Phi)\sigma)(\Phi^{-1}(\Phi(X_t,t)),t) dW_t.
\]
Consequently, the transformed process $P_T(X,W) := (\Phi(X_t,t), W_t)$ solves the equation with transformed coefficients $E_T(\mu) := L(\Phi)\circ (\Phi^{-1},id)$ and $E_T(\sigma) := (D\Phi)\sigma\circ (\Phi^{-1},id)$. Since the Brownian motion and the probability space remain unchanged, diffeomorphisms are classified as strong stochastic transformations.

Differently, the following operations are weak transformations.
Considering a time change $t'=f(t) = \int_0^t f'(s)ds$, if $W_t$ is an adapted Brownian motion, it is a standard result (see, e.g. \cite{oksendal}) that $\tilde{W}_t := \int_0^{f^{-1}(t)} \sqrt{f'(s)} dW_s$ is a Brownian motion with respect to the time-changed filtration $\mathcal{F}_{f^{-1}(t)}$. Using Itô's lemma and the time-change formula for SDEs, the process $X_{f^{-1}(t)}$ satisfies:
\[
d X_{f^{-1}(t)} = \frac{\mu}{f'}\circ (\mathrm{Id},f^{-1}) dt + \frac{\sigma}{\sqrt{f'}}\circ (\mathrm{Id},f^{-1}) d\tilde{W}_t.
\]
Thus, the action on the process is $P_T(X,W) := (X_{f^{-1}(t)}, \tilde{W}_t)$, and the action on the SDE is $E_T(\mu, \sigma) := (\frac{\mu}{f'}\circ (\mathrm{Id},f^{-1}), \frac{\sigma}{\sqrt{f'}}\circ (\mathrm{Id},f^{-1}))$.

Similarly, let $T=B(x,t)$ be an orthogonal matrix-valued function. One can prove that the process $\bar{W}_t := \int_0^t B(X_s,s) dW_s$ remains a Brownian motion by Lévy's characterization. Substituting $dW_t = B^{-1} d\bar{W}_t$ into \eqref{SDE intro} yields:
\[
dX_t = \mu dt + \sigma B^{-1} d\bar{W}_t.
\]
In our framework, this means $P_T(X,W) := (X_t, \bar{W}_t)$ and the coefficients transform as $E_T(\mu, \sigma) := (\mu, \sigma B^{-1})$. Although the probability space is fixed, the modification of the driving noise classifies this as a weak transformation.

Finally, let us introduce a Girsanov transformation defined by $T=h(x,t)$. Under suitable assumptions (see \cite{albeverio}), the process $W'_t = W_t - \int_0^t h(X_s,s)ds$ is a Brownian motion under the new measure $\mathbb{Q}$. Substituting $dW_t = dW'_t - h dt$ into the original equation leads to:
\[
dX_t = (\mu + \sigma h) dt + \sigma dW'_t.
\]
Here, the action on the process is $P_T(X,W) := (X_t, W'_t)$, while the coefficients transform as $E_T(\mu, \sigma) := (\mu + \sigma h, \sigma)$. This is a weak transformation as it modifies both the underlying measure and the driving process.

We can unify these examples into a general definition: a stochastic transformation $T$ is a quadruplet $T=(\Phi,f,B,h)$ encompassing these four operations on the SDE. The global action is obtained by composing these individual effects (see Theorem \ref{trasformate}).\\
\\
Within the set of stochastic transformations, a subclass of particular interest consists of those transformations that preserve the structure of the original SDE. We name \emph{symmetries} these transformations, and they are deeply linked to invariance properties of the underlying diffusion, such as self-similarity (see Example \ref{Example time invariance BM}) or conditioning (see Example \ref{Example BB}).
In the stocastic framework, different notions of invariance yield different definitions of symmetry: we define \emph{strong symmetries} as stochastic transformations that preserve the set of strong solutions; \emph{weak symmetries} as transformations that preserve the set of weak solutions; and \emph{\(\mathcal{G}\)-weak symmetries} as transformations that preserve the law of the solution process (see, \cite{grandemostro},\cite{papernuovo}). By imposing invariance conditions, we can characterize the symmetries of a given SDE through a set of \emph{finite determining equations}, which provide necessary and sufficient conditions for a transformation \(T\) to qualify as a symmetry (see Theorem \ref{equazioni determinanti finite}). 

Using these (finite) determining equations, it is straightforward to verify whether a specific candidate \(T\) is a symmetry for a given SDE. For instance, in the case of Brownian motion, one can easily check that the reflection transformation satisfies the determining equations (see Example \ref{Example reflection}), consistent with the well-known fact that if \(W_t\) is a Brownian motion, then \(-W_t\) is again a Brownian motion.\\
\\
However, the inverse problem is significantly more challenging. Since these equations are highly nonlinear in the components $\Phi, f, B, h$, deriving an unknown symmetry $T$ for a fixed SDE is a difficult task. For this reason, aiming to systematically determine symmetries for general SDEs, we follow the classical deterministic approach (see \cite{Olver}) and look for symmetries in a linearized setting. Indeed, the (possible) transition to a Lie algebra structure would provide significantly more powerful computational tools. But in order to move to a Lie algebra structure, the set of stochastic transformations should constitute a Lie group, that is, a group with a differentiable structure (see \cite{Olver} for a complete discussion). Regarding the group structure, we just need to define the composition between transformations and the  inversion operation. Of course, this could be done in different ways, but we want the group structure to have a meaningful probabilistic counterpart. Thus, given two stochastic transformations $T_1$ and $T_2$, we define $T_1 \circ T_2$ as the unique stochastic transformation under which the transformed solution process ($P_{T_1 \circ T_2}(X)$) and the transformed coefficients ($E_{T_1 \circ T_2}(\mu, \sigma)$) coincide with the iteration of individual transformations (respectively, $P_{T_1}(P_{T_2}(X))$ and $E_{T_1}(E_{T_2}(\mu, \sigma))$).

Once established the algebraic structure, we have now to endow the group of stochastic transformations with a differentiable structure to characterize it as a Lie group. A good way to achieve this goal is to establish a one-to-one correspondence between the group of stochastic transformations and the group of diffeomorphisms of suitable principal bundles. By endowing the latter with an appropriate algebraic structure, this correspondence becomes a group isomorphism. Since the group of diffeomorphisms is a well-known Lie group, the group of stochastic transformations inherits the Lie group structure. In particular, we can consider one-parameter groups of stochastic transformations. The elements of the corresponding Lie algebra are obtained in the standard manner by differentiating the elements of the one-parameter group with respect to the flow parameter and evaluating the result at zero. Intuitively, one can view the Lie group elements as integral curves and the corresponding Lie algebra elements as the tangent vector fields to these curves; formally, the Lie algebra is isomorphic to the tangent space of the Lie group at the identity (see \cite{Olver}). Consequently, we define an \emph{infinitesimal stochastic transformation} as an element of this Lie algebra. Conversely, given an infinitesimal stochastic transformation, it is possible to recover the corresponding finite transformation by reconstructing the flow (see Theorem \ref{teo ricostruzione del flusso}).We can now define \emph{infinitesimal symmetries} as those infinitesimal stochastic transformations that correspond to finite symmetries within the Lie group structure. These symmetries can be now characterized as the solution to \emph{infinitesimal determining equations}, which thus provide necessary and sufficient conditions for an infinitesimal transformation to be an infinitesimal symmetry. This approach overcomes the inherent difficulty of solving finite determining equations directly. Instead, one solves the infinitesimal determining equations—which are significantly easier to handle due to their linearized setting—to find infinitesimal symmetries. Subsequently, the corresponding finite symmetries are recovered by reconstructing the flow.\\
\\
The study of symmetries is of particular interest in determining the invariance properties of diffusion processes, but this is not the only application of this research field. Indeed, in \cite{paper2023}, the authors demonstrated that this geometric approach to studying symmetries leads to the formulation of an integration by parts formula involving the infinitesimal symmetries of a given SDE, inspired by Bismut's approach to Malliavin calculus (\cite{bismut}).

Since originally the integration by parts theorem was derived without including rotations of the Brownian motion in the set of admissible stochastic transformations, in \cite{papernuovo}  the result is generalized to include rotations as well. The theorem states that if $(Y,m,C,H)$ is an infinitesimal symmetry for a given SDE \eqref{SDE intro}---where $Y=\sum_i Y^i \partial_{x_i}$ is a vector field in the Lie algebra corresponding to the diffeomorphism $\Phi$, $m$ is a time function in the Lie algebra corresponding to the time change $f$, $C$ is an antisymmetric matrix corresponding to the rotation matrix $B$, and $H$ is a vector corresponding to the Girsanov drift $h$---then, under suitable regularity assumptions (see Section \ref{Lie symmetries and integration by parts formulas}), the following holds for all sufficiently smooth functions $F$:
\begin{equation}\label{integrazione per parti 1}
    -m(t) \mathbb{E}_{\mathbb{P}}[L(F(X_t))] + \mathbb{E}_{\mathbb{P}}\left[F(X_t) \int_0^t H(X_s,s) \, d W_s\right] + \mathbb{E}_{\mathbb{P}}[Y(F(X_t))] - \mathbb{E}_{\mathbb{P}}[Y(F(X_0))]=0.
\end{equation}
Formula \eqref{integrazione per parti 1} exhibits a clear integration-by-parts structure. In general, it involves zero-order terms ($F(X_t)$), first-order terms ($Y(F(X_t))$, which entails the first derivatives of $F$ as $Y$ is a vector field), and second-order terms ($L(F(X_t))$, where $L$ is the infinitesimal generator of $X$, thus requiring, in principle, both first and second derivatives). Naturally, certain symmetries may have null components, causing specific terms in \eqref{integrazione per parti 1} to vanish; consequently, we may obtain formulas of order zero, one, or two.\\
\\
One of the strongest features of the new Lie symmetry approach to SDEs is the inclusion of Girsanov transformations within the set of admissible transformations. This choice leads to an infinite-dimensional family of symmetries, depending on arbitrary functions of time (see, for instance, the Brownian motion example in Section \ref{Symmetries of BM}). Consequently, when applied to specific examples, formula \eqref{integrazione per parti 1} incorporates these arbitrary time-dependent functions.

A particularly compelling aspect of the integration by parts formula \eqref{integrazione per parti 1} is that by specializing it through suitable choices of these functions, we can recover well-known formulas from probability theory (see Section \ref{Integration by parts and Stein identities}). Among these, Stein's identities are perhaps the most significant. These identities are fundamental tools that characterize probability distributions and serve as the cornerstone of Stein's method for distributional approximation in statistics. However, in the existing literature, such identities are typically derived for specific distributions on a case-by-case basis, and a fully unified approach is still largely missing.

For instance, the classical Stein identity for a Gaussian random variable $Z \sim \mathcal{N}(\mu, \sigma^2)$ is given by $\mathbb{E}[(Z-\mu)F(Z)] = \sigma^2 \mathbb{E}[F'(Z)]$. In the specific case of a Brownian motion at time $t$, where $W_t \sim \mathcal{N}(0,t)$, this identity implies:
\begin{equation}\label{SteinBM target}
    \mathbb{E}[W_t F(W_t)] = t \mathbb{E}[F'(W_t)].
\end{equation}
From our perspective, the emergence of such relations is not surprising. Indeed, although classically derived using different methods, it is important to observe that identity \eqref{SteinBM target} naturally descends from the symmetries of Brownian motion, and in particular from its invariance under Girsanov transformations. This perspective is closely related to Bismut’s formulation of Malliavin calculus. Specifically, one can consider a perturbation of the Brownian motion by a constant drift $h_\lambda = \lambda$. Since $W_t - \lambda t$ is a Brownian motion under the measure $\mathbb{Q}_\lambda$ defined by the Radon-Nikodym derivative $e^{\lambda W_t - \frac{1}{2}\lambda^2 t}$, we have the equality of expectations:
\[
\mathbb{E}_{\mathbb{P}}[F(W_t)] = \mathbb{E}_{\mathbb{Q}_\lambda}[F(W_t - \lambda t)] = \mathbb{E}_{\mathbb{P}}\left[ e^{\lambda W_t - \frac{1}{2}\lambda^2 t} F(W_t - \lambda t) \right].
\]
By differentiating both sides with respect to $\lambda$ and evaluating at $\lambda = 0$, one recovers precisely the target identity \eqref{SteinBM target}.

In our framework, formula \eqref{integrazione per parti 1}, being much more general, captures this relationship directly. We obtain the identity as a particular case by specializing the arbitrary time-dependent function appearing in the formula to the specific choice corresponding to the Girsanov symmetry used in the argument above.
Indeed, consider a one-dimensional Brownian motion $dX_t=dW_t$. In this case, formula \eqref{integrazione per parti 1} ensures that for any time-dependent function $\beta(t)$ and any bounded function $F \in C^2_b$, the following holds:
\begin{equation}\label{Vbeta_intro 1d}
    \mathbb{E}_{\mathbb{P}}\left[F(W_t)\int_0^t -\beta'(s) \, dW_s\right] + \mathbb{E}_{\mathbb{P}}[\beta(t)F'(W_t)] - \mathbb{E}_{\mathbb{P}}[\beta(0)F'(W_0)] = 0.
\end{equation}
By setting $\beta(t) = t$, identity \eqref{Vbeta_intro 1d} reduces to:
\[
    \mathbb{E}_{\mathbb{P}}[W_t F(W_t)] = t \mathbb{E}_{\mathbb{P}}[F'(W_t)].
\]
As expected, this result coincides exactly with the Brownian Stein identity \eqref{SteinBM target}.\\
Stein's identity is not the only classical result recoverable from \eqref{integrazione per parti 1}. By varying the choices of parameters, we also obtain other well-known invariance properties: Isserlis' theorem, Lévy's stochastic area, and much more (see Section \ref{Integration by parts and Stein identities}). Indeed, classical identities of Gaussian analysis emerge here as an unified manifestations of the underlying symmetry groups associated with the Wiener process. The flexibility in choosing temporal parameters suggests that the integration by parts theorem acts as a generating mechanism for an infinite family of conservation laws for the specific diffusion at hand.

While, for simplicity, Section \ref{Integration by parts and Stein identities} of this paper focuses exclusively on Brownian motion, we stress that the integration by parts formula can be derived for generic diffusion processes, yielding equally compelling results.

\section{Finite stochastic transformations}\label{Finite stochastic transformations}
Throughout this paper, we denote by $\mathrm{SDE}_{\mu,\sigma}$ the following equation:
\begin{equation}\label{SDEms}
    dX_t = \mu(X_t,t) \, dt + \sigma(X_t,t) \, dW_t,
\end{equation}
where $W$ is an $d$-dimensional Brownian motion, $\mu(x,t) \in \mathbb{R}^n$, and $\sigma(x,t) \in \mathbb{R}^{n \times d}$. We denote by $DF$ the derivative of a function $F$, whose interpretation depends on the dimension: for $F: \mathbb{R} \to \mathbb{R}$, it represents the first derivative; for $F: \mathbb{R}^n \to \mathbb{R}$, it denotes the gradient; and for $F: \mathbb{R}^n \to \mathbb{R}^d$, it corresponds to the Jacobian matrix. Similarly, we denote by $D^2 F$ the second derivative (representing the standard second derivative or the Hessian matrix, depending on the dimension).\\
A stochastic transformation \( T \) acts simultaneously on a given SDE and on its solution process \( (X, W) \). Following the notation in \cite{paper2015, paper2020, paper2023}, we denote by \( E_T \) the action on the SDE coefficients and by \( P_T \) the action on the solution process. This structure ensures consistency: if \( (X, W) \) is a (weak) solution to \( SDE_{\mu, \sigma} \), then the transformed process \( P_T(X, W) := (P_T(X), P_T(W)) \) is a (weak) solution to the transformed equation \( E_T(SDE_{\mu, \sigma}) := SDE_{E_T(\mu), E_T(\sigma)} \).

We distinguish between \emph{strong} and \emph{weak} stochastic transformations. A strong transformation preserves the Brownian motion (i.e., \( P_T(W) = W \)) and the underlying probability space. Conversely, a weak transformation modifies the Brownian component (\( P_T(W) \neq W \)) or the filtered probability space; in this latter case, well-posedness requires verifying that \( P_T(W) \) remains a Brownian motion, possibly under a new filtration or probability measure.

\begin{definizione}\label{def trasf}
    Let \( M, M' \) be open subsets of \( \mathbb{R}^n \). Consider  \( SDE_{\mu,\sigma} \) \eqref{SDEms}, driven by an \( m \)-dimensional Brownian motion \( W\).
    A (finite) stochastic transformation from \( M \) to \( M' \) is a quadruple \( T=(\Phi, f, B, h) \) of smooth functions, whose components act as follows:
    
    \begin{itemize}
       
        \item \textbf{Spatial diffeomorphism (\(\Phi\)):} The function \( \Phi: M \times \mathbb{R}_+ \to M' \) defines a time-dependent spatial coordinate change \( T_1 \). It leaves the probability space unchanged. Its action is:
        \[ 
        P_{T_1}(X)_t = \Phi(X_t,t), \quad P_{T_1}(W)_t = W_t, 
        \]
        \[ 
        E_{T_1}(\mu) = L_t(\Phi)\circ (\Phi^{-1},\mathrm{id}), \quad E_{T_2}(\sigma) = (D(\Phi) \sigma)\circ (\Phi^{-1},\mathrm{id}), 
        \]
        where \( L_t = \partial_t + \sum\limits_{i=1}^n\mu_i \partial_i + \sum\limits_{i,j=1}^n\frac{1}{2} (\sigma \sigma^T )_{ij}\partial_{ij} \) is the infinitesimal generator of \( SDE_{\mu,\sigma}\).
        
        \item \textbf{Time change (\(f\)):} The strictly positive function \( f': \mathbb{R}_+ \to \mathbb{R}_+ \) defines an absolutely continuous time change \( f(t) := \int_0^t f'(s)ds \). This transformation, denoted by \( T_2 \), changes the filtration from \( \mathcal{F}_t \) to \( \mathcal{F}_{f^{-1}(t)} \). Its action is:
        \[ 
        P_{T_2}(X)_t = X_{f^{-1}(t)}, \quad P_{T_2}(W)_t = \int_0^{f^{-1}(t)}\sqrt{f'(s)}dW_s, 
        \]
        \[ 
        E_{T_2}(\mu) = \frac{\mu}{f'}\circ (\mathrm{Id},f^{-1}), \quad E_{T_2}(\sigma) = \frac{\sigma}{\sqrt{f'}}\circ (\mathrm{Id},f^{-1}). 
        \]

        \item \textbf{Random rotation (\(B\)):} The function \( B: M \times \mathbb{R}_+ \to SO(d) \) defines a rotation of the Brownian motion \( T_3 \). It preserves the probability space. Its action is:
        \[ 
        P_{T_3}(X) = X, \quad P_{T_3}(W)_t = \int_0^t B(X_s,s) dW_s, 
        \]
        \[ 
        E_{T_3}(\mu) = \mu, \quad E_{T_3}(\sigma) = \sigma B^{-1}. 
        \]

        \item \textbf{Drift change/Girsanov (\(h\)):} The function \( h: M \times \mathbb{R}_+\to \mathbb{R}^d \) defines a transformation \( T_4 \) that changes the measure from \( \mathbb{P} \) to \( \mathbb{Q} \) via the Radon-Nikodym derivative:
        \[ 
        \frac{d \mathbb{Q}}{d \mathbb{P}}\Big|_{\mathcal{F}_t} = \exp \left( \int_0^t h(X_s,s) \cdot dW_s - \frac{1}{2} \int_0^t |h(X_s,s)|^2 ds \right). 
        \]
        Assuming regularity conditions to ensure non-explosivity (see \cite{paper2020}), the action of \( T_4\) is:
        \[ 
        P_{T_4}(X) = X, \quad P_{T_4}(W)_t = W_t - \int_0^t h(X_s,s)dt, 
        \]
        \[ 
        E_{T_4}(\mu) = \mu + \sigma h, \quad E_{T_4}(\sigma) = \sigma. 
        \]
    \end{itemize}
    For each $i=1,\dots,4$, the transformation $T_i$ is well-defined. Specifically, the transformed driving Brownian motion, $P_{T_i}(W)$, remains a Brownian motion with respect to the filtered probability space induced by $T_i$. For the non-trivial cases, we refer to~\cite{oksendal} for time change and to~\cite{baldi} for rotation and change of measure.
\end{definizione}

The general stochastic transformation is obtained by composing these elementary actions.

\begin{teorema}[Double action of the transformation]\label{trasformate}
    Let \( (X,W)\) be a (weak) solution to \( SDE_{\mu,\sigma}\) and let \( T=(\Phi, f, B, h) \) be a stochastic transformation. Then, the transformed process \( P_T(X,W) \) is a solution to \( E_T(SDE_{\mu,\sigma}) \).
    
    The transformed process \( P_T(X,W):=(P_T(X),P_T(W))\)is defined by:
    \begin{align*}
        P_{T}(X)_t &= \Phi(X_{f^{-1}(t)},f^{-1}(t)), \\
        P_{T}(W)_t &= \int_0^{f^{-1}(t)}\sqrt{f'(s)}B(X_s,s)(dW_s - h(X_s,s) ds),
    \end{align*}
    where \( P_T(W) \) is a Brownian motion with respect to the filtration \( \mathcal{F}'_t=\mathcal{F}_{f^{-1}(t)} \) and the probability measure \( \mathbb{Q} \) defined by the density:
    \[ 
    \frac{d \mathbb{Q}}{d \mathbb{P}}\Big|_{\mathcal{F}_t} = \exp \left( \int_0^t h(X_s,s) \cdot dW_s - \frac{1}{2} \int_0^t |h(X_s,s)|^2 ds \right).
    \]
    The transformed coefficients of \( E_T(SDE_{\mu,\sigma}):=SDE_{E_T(\mu), E_T(\sigma)}\) are given by:
    \begin{align*}
        E_{T}(\mu) &= \left[ \frac{1}{f'} \big( L_t(\Phi) + D(\Phi) \sigma h \big) \right] \circ (\Phi^{-1},f^{-1}), \\
        E_{T}(\sigma) &= \left[ \frac{1}{\sqrt{f'}} D(\Phi) \sigma B^{-1} \right] \circ (\Phi^{-1},f^{-1}).
    \end{align*}
\end{teorema}

\begin{proof}
    The result follows by iterating the actions of the four components described in Definition \ref{def trasf}. See \cite{paper2020} for details.
\end{proof}

\begin{osservazione}
    The distinction between strong and weak transformations mirrors that of strong and weak solutions. A \emph{strong} stochastic transformation acts on a fixed probability space and filtration, leaving the Brownian motion unchanged. This requires \( h=0 \) (no measure change), \( f=\mathrm{id} \) (no change in the filtration), and \( B=I_d \) (no change in the driving Brownian motion). Thus, a transformation is strong if and only if \( T=(\Phi, \mathrm{id}, I_d, 0) \).
\end{osservazione}

\begin{osservazione}
    Any elementary transformation can be viewed as a subclass of the general tuple \( T \):
    a diffeomorphism corresponds to  \( (\Phi, \mathrm{id}, I_d, 0) \)), a time change to \( (\mathrm{Id}, f, I_d, 0) \), a random rotation to \( (\mathrm{Id}, \mathrm{id}, B, 0) \), and a measure change to \( (\mathrm{Id}, \mathrm{id}, I_d, h) \).
\end{osservazione}

\subsection{From Brownian motion to Geometric Brownian motion via strong transformation}
%Geometric Brownian motion as a strong transformation of Brownian motion}
\label{Geometric Brownian motion}
Consider a one-dimensional Brownian motion \( dX_t = dW_t \), with drift \( \mu_0:=0\) and diffusion coefficient \( \sigma_0:=1\). Applying the strong transformation \( T=(\Phi, \mathrm{id}, 1, 0) \) defined by
\[
\Phi(x,t) = z_0 \exp\Big( (\mu - \tfrac{1}{2}\sigma^2)t + \sigma x \Big)
\]
yields the Geometric Brownian Motion (GBM). Indeed, by Theorem \ref{trasformate}, the transformed process \( Z_t := P_T(X)_t = \Phi(X_t, t) \) solves the SDE determined by the transformed coefficients \( \bar{\mu}:=E_T(\mu_0) \) and \( \bar{\sigma}:=E_T(\sigma_0) \).

Computing the action of the generator \( L_t = \partial_t + \frac{1}{2}\partial_{xx} \) and the derivative \( D = \partial_x \) on \( \Phi \), we obtain:
\begin{align*}
    L_t(\Phi) &= \partial_t \Phi + \tfrac{1}{2} \partial_{xx} \Phi =  (\mu - \tfrac{1}{2}\sigma^2)\Phi + \tfrac{1}{2}  \sigma^2 \Phi = \mu \Phi, \\
    D(\Phi) &= \sigma \Phi.
\end{align*}
The transformed coefficients are obtained by composing with \( \Phi^{-1} \) (see Definition \ref{def trasf}):
\[
\bar{\mu}(z,t):=E_T(\mu_0)(z,t) = L_t(\Phi)\circ \Phi^{-1} = \mu z, \qquad \bar{\sigma}(z,t):=E_T(\sigma_0)(z,t) = D(\Phi)\circ \Phi^{-1} = \sigma z.
\]
Thus, \( T\) yields the standard geometric Brownian motion equation:
\begin{equation}\label{GBM}
    dZ_t = \mu Z_t \, dt + \sigma Z_t \, dW_t, \quad Z_0 = z_0.
\end{equation}
Since \( T \) is strong, the underlying filtered probability space remains unchanged. The generalization to the \( n \)-dimensional case is straightforward and follows the same logic component-wise.

\subsection{From Brownian motion to Brownian Bridge via weak transformation }
%Brownian Bridge as a weak transformation of Brownian motion}
\label{Example BB}

 Consider an \( n \)-dimensional Brownian motion \( d\underline{X}_t=d\underline{W}_t \), with \( \underline{X}=(X^1,\cdots,X^n)^T\), \( \underline{W}=(W^1,\cdots,W^n)^T\), drift \( \mu_0:=\underline{0}\) and diffusion coefficient \( \sigma_0:=I_n\). Let us apply the weak stochastic transformation \( T=(\Phi,f,I_n,h) \) defined by:
\begin{equation*}
    f(t)=\frac{t T^2}{1+tT}, \quad \Phi(\underline{x},t)=(T-f(t))\underline{x}, \quad h(\underline{x},t)=\frac{f'(t)}{T-f(t)}\underline{x}.
\end{equation*}
According to Theorem \ref{trasformate}, the transformed process \( P_T(\underline{X}, \underline{W}) \) solves the SDE with coefficients \( E_T(\mu, \sigma) \).
The transformed solution process is given by:
\begin{equation*}
    \underline{Z}_t := P_T(\underline{X})_t = \Phi(\underline{X}_{f^{-1}(t)}, f^{-1}(t)) = (T-t)\underline{X}_{f^{-1}(t)}.
\end{equation*}
The transformed driving noise is:
\begin{equation}\label{W bb}
    P_T(\underline{W})_t = \int_0^{f^{-1}(t)} \sqrt{f'(s)}\, (d\underline{W}_s - h(\underline{X}_s,s)\,ds) 
    = \tilde{W}_t - \int_0^t \frac{\underline{Z}_s}{T-s}\,ds,
\end{equation}
where \( \tilde{W}_t := \int_0^{f^{-1}(t)} \sqrt{f'(s)}\,dW_s \) is a Brownian motion under the transformed filtration (according to Definition \ref{def trasf}).

We now compute the transformed coefficients \( E_T(\mu_0, \sigma_0) \). The new drift is:
\begin{equation*}
    \mu' := E_T(\mu_0) = \Big[ \frac{1}{f'}\Big(L_t(\Phi)+D\Phi \cdot h\Big)\Big]\circ (\Phi^{-1}, f^{-1})
    = \Big[ \frac{1}{f'}\Big(-f'(t) \underline{x} + (T-f(t))\frac{f'(t)}{T-f(t)}\underline{x}\Big)\Big]\circ (\Phi^{-1}, f^{-1}) = \underline{0}.
\end{equation*}
The new diffusion coefficient is:
\begin{equation*}
    \sigma' := E_T(\sigma_0) = \Big[\frac{1}{\sqrt{f'}}(D\Phi) I_n \Big]\circ (\Phi^{-1}, f^{-1}) 
    = \Big[\frac{T-f(t)}{T-f(t)} I_n \Big]\circ (\Phi^{-1}, f^{-1}) = I_n.
\end{equation*}
Consequently, the transformed process satisfies the SDE:
\begin{equation*}
 d\underline{Z}_t = \mu' dt + \sigma' dP_T(\underline{W})_t = dP_T(\underline{W})_t.
\end{equation*}
Substituting the expression for \( P_T(\underline{W})_t \) from \eqref{W bb}, we obtain:
\begin{equation*}
    d\underline{Z}_t = -\frac{\underline{Z}_t}{T-t}\,dt + d\tilde{W}_t.
\end{equation*}
Thus, the transformation \( T \) applied to a standard Brownian motion yields a \emph{Brownian Bridge}. 

It is worth noting that since \( E_T(\mu_0)=\mu_0 \) and \( E_T(\sigma_0)=\sigma_0 \), the structural form of the SDE is preserved: \( \underline{Z}_t \) solves a Brownian motion SDE driven by \( P_T(\underline{W}) \). However, \( P_T(\underline{W}) \) is a Brownian motion only under the probability measure \( \mathbb{Q} \) defined by the Doléans-Dade exponential associated with \( h \). This reflects the fact that a Brownian Bridge can be viewed as a Brownian motion conditioned on reaching \( 0 \) at time \( T \) (see \cite{ChetriteTouchette2015}).

We refer to such transformations—those preserving the form of the SDE coefficients—as \textit{symmetries}. These are intrinsically linked to the invariance properties of the diffusion process and will be discussed extensively in the following section.

\section{Invariance properties and finite symmetries}\label{Invariance properties and finite symmetries}

Example \ref{Example BB}  shows that specific transformations can preserve the structure of an SDE, reflecting intrinsic properties of the underlying diffusion such as  conditioning or self-similarity (see Example \ref{Example time invariance BM}). We define \emph{symmetries} as stochastic transformations that leave the original SDE invariant. Following the framework in \cite{paper, paper2020, paper2023}, an SDE is (strongly/weakly) invariant if the transformation \(T\) preserves the set of its (strong/weak) solutions. A broader notion, introduced in \cite{papernuovo}, requires only the preservation of the solution's law, linking symmetry to the invariance of an entire family of SDEs rather than of a single equation.

\begin{definizione}\label{def symmetries}
Consider the stochastic differential equation \(SDE_{\mu,\sigma}\)
\[
dX_t = \mu(X_t,t)\,dt + \sigma(X_t,t)\,dW_t
\]
on the filtered probability space \( (\Omega, \mathcal{F}, (\mathcal{F}_t)_t, \mathbb{P})\).

\begin{itemize}
    \item A strong stochastic transformation \(T = (\Phi, \mathrm{id}, I_d, 0)\) is a \emph{strong (finite) symmetry} if, for every strong solution \(X\) with driving Brownian motion \(W\), the transformed process \(P_T(X)\) remains a solution to \(SDE_{\mu,\sigma}\) in the same space driven by \(P_T(W)=W\). Being composed of diffeomorphisms, \(T\) modifies neither the probability space nor the Brownian motion.
    
    \item A weak stochastic transformation \(T = (\Phi, f, B, h)\) is a \emph{weak (finite) symmetry} if, for every weak solution \((X,W)\), the transformed pair \(P_T(X,W)\) remains a weak solution to \(SDE_{\mu,\sigma}\) in \( (\Omega, \mathcal{F}, (\mathcal{F'}_t)_t, \mathbb{Q})\), where \(\mathbb{Q}\) and \(\mathcal{F'}_t\) are the measure and the filtration induced by \(T\), respectively.

    \item A weak stochastic transformation \(T = (\Phi, f, B, h)\) is a \(\mathcal{G}\)-weak symmetry if it preserves the solution set of the martingale problem associated with \((\mu, \sigma\sigma^T)\). Equivalently, for every weak solution \((X,W)\), the transformed process \(P_T(X)\) has the same law as \(X\).
\end{itemize}
\end{definizione}

Imposing these invariance conditions characterizes symmetries via the coefficients of the SDE.

\begin{teorema}\label{equazioni determinanti finite}
Let \(SDE_{\mu,\sigma}\) be defined as above.

\begin{itemize}
    \item \(T = (\Phi, \mathrm{id}, I_d, 0)\) is a strong symmetry if and only if
    \[
    \mu = L_t(\Phi)\circ(\Phi^{-1}, \mathrm{id}), \qquad \sigma = \big(D\Phi \cdot \sigma\big)\circ(\Phi^{-1},\mathrm{id}).
    \]

    \item \(T = (\Phi, f, B, h)\) is a weak symmetry if and only if
    \[
    \mu = \Big(\frac{1}{f'}\big[\,L_t(\Phi) + D\Phi\,\sigma h\,\big]\Big)\circ(\Phi^{-1},f^{-1}), \qquad \sigma = \Big(\frac{1}{\sqrt{f'}}\, D\Phi\, \sigma B^{-1}\Big)\circ(\Phi^{-1},f^{-1}).
    \]

    \item \(T = (\Phi, f, B, h)\) is a \(\mathcal{G}\)-weak symmetry if and only if
    \[
    \mu = \Big(\frac{1}{f'}\big[\,L_t(\Phi) + D\Phi\,\sigma h\,\big]\Big)\circ(\Phi^{-1},f^{-1}), \qquad \sigma\sigma^T = \Big(\frac{1}{f'}\, D\Phi\, \sigma\sigma^T\, D\Phi^T\Big)\circ(\Phi^{-1},f^{-1}).
    \]
\end{itemize}
\end{teorema}

\begin{proof}
Proofs for the first two statements are in \cite{paper2020}, and the third in \cite{papernuovo}. The logic relies on Theorem~\ref{trasformate}, which states that \(P_T(X,W)\) solves the transformed equation \(SDE_{E_T(\mu),\,E_T(\sigma)}\).
\begin{itemize}
    \item For strong and weak symmetries, invariance requires the transformed coefficients to match the original ones: \(\mu = E_T(\mu)\) and \(\sigma = E_T(\sigma)\). Substituting the explicit forms of \(E_T\) yields the stated equations.
    \item For \(\mathcal{G}\)-weak symmetries, preserving the law implies preserving the infinitesimal generator \(L\). Thus, we require \(\mu = E_T(\mu)\) and \(\sigma\sigma^T = E_T(\sigma)E_T(\sigma)^T\).
\end{itemize}
\end{proof}

The relations in Theorem~\ref{equazioni determinanti finite} are called \emph{finite determining equations}. They form a highly nonlinear system. To overcome this complexity, we turn to infinitesimal transformations. This approach linearizes the problem: instead of solving nonlinear functional equations for \(T\), we determine the generator of a one-parameter group of transformations. This motivates the algebraic structure introduced in Section \ref{Algebraic structure of stochastic transformations}.

\begin{osservazione}\label{equivalenza wgw}
The symmetry notions in Definition~\ref{def symmetries} follow the hierarchy $\text{Strong} \subset \text{Weak} \subset \mathcal{G}\text{-weak}$. Strong symmetries are weak symmetries which, in particular, preserve the driving Brownian motion and the filtered probability space. Weak symmetries are $\mathcal{G}$-weak as they preserve the law. The converses generally fail: a weak symmetry is strong only if \( f=\mathrm{id}\), $B = I_d$, and $h = 0$, while a $\mathcal{G}$-weak symmetry is not necessarily weak unless \( E_T(\sigma)=\sigma\). This is because the determining equations for $\mathcal{G}$-weak symmetries do not constrain the rotation matrix ($B$), making any rotation a trivial $\mathcal{G}$-weak symmetry: if \( T=(\mathrm{Id},\mathrm{id}, B,0)\) is a rotation, then from Definition \ref{def trasf} \( E_T(\sigma)=\sigma B^{-1}\); since \( B\) is a rotation matrix, by construction \( B B^T=I_d\), so it is immediate to verify that \( E_T(\sigma) E_T(\sigma)^T=\sigma \sigma^T\).

Equivalence between weak and $\mathcal{G}$-weak symmetries holds only under structural conditions, such as the invertibility of $\sigma \sigma^{T}$. In this case, decomposing and using Theorem \ref{equazioni determinanti finite}
\[
\sigma = \sigma\sigma^T \sigma (\sigma^T\sigma)^{-1}=\left( \tfrac{1}{f'} D(\Phi) \sigma \sigma^{T} D(\Phi)^{T} \right) \circ (\Phi^{-1},f^{-1})\cdot  \sigma (\sigma^{T}\sigma)^{-1},
\]
we find that to satisfy the weak symmetry condition, $B$ is uniquely determined by
\[
B^{-1} \circ \Phi^{-1} = \left( \tfrac{1}{\sqrt{f'}} \sigma^{T} D(\Phi)^{T} \right) \circ (\Phi^{-1},f^{-1})\cdot  \sigma (\sigma^{T}\sigma)^{-1},
\]
which yields an orthogonal matrix. See \cite{papernuovo} for the connection to the ``square root of a matrix field'' problem.
\end{osservazione}

\begin{osservazione}
If $X$ satisfies $dX_t=\mu dt+\sigma dW_t$ and $B \in SO(d)$, then $X$ is also a weak solution to $dX_t=\mu dt+(\sigma B^{-1}) dW'_t$, driven by $W'_t = \int_0^t B dW_s$. The inclusion of rotations implies that the law of a weak solution is no longer identified by the single \( SDE_{\mu,\sigma}\), but by the whole equivalence class $\mathcal{Gauge}(SDE_{\mu,\sigma}) := \{ SDE_{\mu,\sigma B^{-1}} \mid B \in SO(d) \}$. Consequently, $\mathcal{G}$-weak symmetries are transformations preserving the solution set of $\mathcal{Gauge}(SDE_{\mu,\sigma})$. See \cite{papernuovo, albeverio} for the connection with gauge symmetries.
\end{osservazione}

\subsection{Reflection of Brownian motion}\label{Example reflection}

It is well known that if $\underline{W}_t$ is an $n$-dimensional Brownian motion, then $\underline{Y}_t := -\mathbf{W}_t$ is also a Brownian motion on the same filtered probability space (see \cite{baldi}). We aim to recover this reflection as a symmetry of the SDE $d\underline{X}_t = d\underline{W}_t$, where $\mu = \underline{0}$ and $\sigma = I_n$.

First, consider the transformation $T = (\Phi, \mathrm{id}, I_n, 0)$ with $\Phi(\underline{x}) = -\underline{x}$. This map involves only a diffeomorphism of the state space, implying $\Phi(\underline{X}_t) = -\underline{W}_t$. According to Theorem \ref{trasformate}, the coefficients of the transformed SDE are:
\[
E_T(\mu) =L_t(\Phi)\circ\Phi^{-1}= \tfrac{1}{2} D^2\Phi \circ \Phi^{-1} = \underline{0} = \mu, \qquad 
E_T(\sigma) = (D\Phi \cdot \sigma) \circ \Phi^{-1} = -I_n.
\]
While $E_T(\sigma) \neq \sigma$, we observe that $E_T(\sigma) E_T(\sigma)^T = (-I_n)(-I_n)^T = I_n = \sigma \sigma^T$. Thus, $T$ is a $\mathcal{G}$-weak symmetry for Definition \ref{def symmetries}. In this context, $T$ preserves the law of the solution process but fails to preserve the SDE structure, as it interprets the reflected process via $d\underline{Y}_t = -d\underline{W}_t$.

Since $\sigma\sigma^T = I_n$ is invertible, Remark \ref{equivalenza wgw} ensures the equivalence between weak and $\mathcal{G}$-weak symmetries. To recover the weak symmetry $T_1 = (\Phi, B, 1, 0)$, following the Remark \ref{equivalenza wgw} we determine the orthogonal matrix $B$ such that:
\[
B^{-1} \circ \Phi^{-1} = (\sigma^T D\Phi^T) \circ \Phi^{-1} \cdot \sigma (\sigma^T \sigma)^{-1} = -I_n,
\]
yielding $B = -I_n$. The resulting transformation $T_1 = (\Phi, \mathrm{id}, -I_n, 0)$ maps $\underline{X}$ to $-\underline{W}$ and preserves the SDE coefficients:
\[
E_{T_1}(\mu) = \mu, \qquad E_{T_1}(\sigma) = (D\Phi \cdot \sigma B^{-1}) \circ \Phi^{-1} = (-I_n)(I_n)(-I_n) = I_n = \sigma.
\]
Thus, $T_1$ is a weak symmetry according to Definition \ref{def symmetries}. It interprets the reflected process as $d\underline{Y}_t = d(-{\underline{W}}_t)$.

\begin{osservazione}
Both $T$ and $T_1$ are \emph{discrete} symmetries, since they cannot be generated by the flow of an infinitesimal generator. Consequently, they are recovered only as finite transformations (see \cite{Olver}).
\end{osservazione}

\section{Algebraic structure of stochastic transformations}\label{Algebraic structure of stochastic transformations}
We aim to endow the set of stochastic transformations with a group structure. Naturally, one could define composition and inversion in several different ways; however, the resulting algebraic structure should satisfy two key requirements:
\begin{enumerate}
    \item It must rest on a meaningful probabilistic foundation;
    \item It must allow for the formulation of an associated Lie algebra, thereby providing more powerful computational tools in a suitably linearized setting.
\end{enumerate}

\subsection{Probabilistic foundation of the algebraic structure}\label{probabilistic foundation of the algebraic structure}

The algebraic structure of stochastic transformations must be grounded in a well-defined probabilistic interpretation. Specifically, we require that the composition of transformations corresponds to the composition of the induced processes on solutions and coefficients. That is, for any \(T_1, T_2\), the composed transformation \(T_2 \circ T_1\) is the unique transformation satisfying:
\[
P_{T_2\circ T_1}(X,W) = P_{T_2}\!\left(P_{T_1}(X,W)\right)
\quad \text{and} \quad
E_{T_2\circ T_1}(\mu,\sigma)=E_{T_2}\!\left(E_{T_1}(\mu,\sigma)\right).
\]
Let \( T_1=(\Phi_1, f_1, B_1, h_1)\) and \( T_2=(\Phi_2, f_2, B_2, h_2)\). Assuming \( T_2\circ T_1=(\tilde{\Phi}, \tilde{f}, \tilde{B},\tilde{h})\), we compare the transformed solution process \(X\) and the driving noise \(W\).
First, for the process \(X\), we have:
\begin{equation}\label{P_T2T1}
    P_{T_2\circ T_1}(X)=\tilde{\Phi}(X_{\tilde{f}^{-1}(t)},\tilde{f}^{-1}(t)).
\end{equation}
Conversely, applying the transformations sequentially:
\begin{equation}\label{PT2PT1}
    P_{T_2}(P_{T_1}(X))_t = (\Phi_2\circ\Phi_1)(X_{(f_2\circ f_1)^{-1}(t)}, (f_2\circ f_1)^{-1}(t)).
\end{equation}
Comparing \eqref{P_T2T1} and \eqref{PT2PT1} identifies the spatial and temporal components:
\[
\tilde{\Phi}=\Phi_2\circ \Phi_1, \qquad \tilde{f}=f_2\circ f_1.
\]
Similarly, for the Brownian motion, the composed transformation yields
\begin{equation}\label{PT2T1W}
    P_{T_2\circ T_1}(W)=\int_0^{\tilde{f}^{-1}(t)}\sqrt{\tilde{f'}(s)}\tilde{B}(X_s)(dW_s - \tilde{h}(X_s,s) ds),
\end{equation}
while the sequential application gives
\begin{equation}\label{PT2PT1W}
\begin{split}
    P_{T_2}(P_{T_1}(W)) = \int_0^{(f_2\circ f_1)^{-1}(t)} &\sqrt{[(f'_2\circ f_1) f'_1](s)} \, [ (B_2\circ\Phi_1) B_1](X_s,s) \\
    &\cdot \Bigg( dW_s - \Big(h_1(X_s,s)+\sqrt{f'_1(s) }B_1^{-1}(X_s,s) (h_2 \circ \Phi_1)(X_s,s) \Big) ds \Bigg).
\end{split}
\end{equation}
Matching the terms in \eqref{PT2T1W} and \eqref{PT2PT1W} determines 
\[\tilde{B}=(B_2\circ\Phi_1) B_1, \qquad \tilde{h}=h_1+\sqrt{f_1'}B_1^{-1}\cdot (h_2\circ(\Phi_1,\mathrm{id})).\]  It is easy to verify that with these choices for the components of \(\tilde{T}\), we also have that \( E_{T_2\circ T_1}(\mu,\sigma)\equiv E_{T_2}(E_{T_1}(\mu,\sigma)).\)

To define a group structure, it remains to establish the neutral element and the inverse.
The neutral element \(T_0 = (\Phi_0, f_0, B_0, h_0)\) must satisfy \(T \circ T_0 = T = T_0 \circ T\), which yields
\begin{equation}\label{neutro}
\Phi_0 = \mathrm{Id}, \quad f_0=\mathrm{id} \quad B_0 = I_n, \quad h_0 = 0.
\end{equation}
The inverse \(T^{-1} = (\bar{\Phi}, \bar{f}, \bar{B}, \bar{h})\) is defined by \(T \circ T^{-1} = T^{-1}\circ T = T_0\). Solving the composition equations for the barred components results in:
\begin{equation}\label{inverso}
T^{-1} = \left( \Phi^{-1}, \; f^{-1}, \; (B\circ(\Phi^{-1},\mathrm{id}))^{-1}, \;  -\frac{1}{\sqrt{f'}} \, (B \cdot h\circ(\Phi^{-1},\mathrm{id})) \right).
\end{equation}
We summarize these results in the following definition.

\begin{definizione}\label{group structure}
The set of stochastic transformations forms a group with the following operations:
\begin{itemize}
    \item \textbf{Composition:} Given \( T_1=(\Phi_1, f_1, B_1, h_1)\) and \( T_2=(\Phi_2, f_2, B_2, h_2)\),
    \[
    T_2\circ{T_1} = \Big(\Phi_2\circ\Phi_1,\; f_2\circ f_1, \; (B_2\circ(\Phi_1,\mathrm{id}))\cdot{B_1},\;  h_1 + \sqrt{f'_1}B_1^{-1}\cdot(h_2\circ(\Phi_1,\mathrm{id})) \Big).
    \]
    \item \textbf{Inverse:} Given \( T=(\Phi, f, B, h) \),
    \[
    T^{-1} = \Big(\Phi^{-1},\; f^{-1}, \; (B\circ\Phi^{-1})^{-1},\;  -\frac{1}{\sqrt{f'}}(B\cdot{h}\circ(\Phi^{-1},\mathrm{id})) \Big).
    \]
    \item \textbf{Neutral Element:} \( T_0=(\mathrm{Id}, \mathrm{id}, I_n, 0) \).
\end{itemize}
\end{definizione}

\subsubsection{Geometric Brownian motion (again)}

We previously obtained the GBM via a strong transformation (Example~\ref{Geometric Brownian motion}). Here, we show it can also be constructed from a standard Brownian motion via a \emph{weak} transformation involving a change of probability measure. In this example, we will also illustrate how the group structure described in Definition \ref{group structure} is well-defined in terms of process composition and transformed coefficients.

Let the target GBM on \((\Omega, \mathcal{F}, (\mathcal{F}_t)_t, \mathbb{Q})\) be
\begin{equation}\label{GBM1}
dY_t = \mu Y_t\, dt + \sigma Y_t\, d\bar{W}_t,
\end{equation}
and the source standard Brownian motion on \((\Omega, \mathcal{F}, (\mathcal{F}_t)_t, \mathbb{P})\) be
\begin{equation}\label{BM1}
dX_t = dW_t, \quad (\text{i.e., } \mu_0 = 0, \sigma_0 = 1).
\end{equation}
Consider the weak transformation \(T = (\Phi, I_n, 1, h)\) defined by
\[
\Phi(x) = e^{\sigma x}, \qquad h = \frac{\mu}{\sigma} - \frac{\sigma}{2}.
\]

Applying Theorem~\ref{trasformate}, the transformed process \( (P_T(X), P_T(W)) \) is given by
\[
Y_t := \Phi(X_t) = e^{\sigma X_t}, \qquad \bar{W}_t := W_t - \int_0^t h\, ds = W_t - \left(\frac{\mu}{\sigma} - \frac{\sigma}{2}\right)t.
\]
The change of measure defining \(\mathbb{Q}\) is
\[
\frac{d\mathbb{Q}}{d\mathbb{P}}\Big|_{\mathcal{F}_t}
= \exp\left( \int_0^t h\, dW_s - \frac{1}{2}\int_0^t h^2\, ds \right)
= \exp\left( \left(\frac{\mu}{\sigma} - \frac{\sigma}{2}\right)W_t - \frac{1}{2}\left(\frac{\mu}{\sigma} - \frac{\sigma}{2}\right)^2 t \right).
\]
The transformed coefficients \( E_T(\mu_0, \sigma_0) \) are computed as follows:
\begin{align*}
\mu'(x) := E_T(\mu_0)&=\left(\tfrac{1}{2}D^2\Phi + D\Phi \cdot h \right)\circ \Phi^{-1}(x) 
= \left(\tfrac{1}{2}\sigma^2 e^{\sigma (\cdot)} + \sigma e^{\sigma (\cdot)}\left(\frac{\mu}{\sigma} - \frac{\sigma}{2}\right)\right)\circ \Phi^{-1}(x) = \mu x, \\
\sigma'(x) := E_T(\sigma_0) &=(D\Phi)\circ \Phi^{-1}(x) = (\sigma e^{\sigma (\cdot)})\circ \Phi^{-1}(x) = \sigma x.
\end{align*}
Thus, \(P_T(X,W)\) solves \eqref{GBM1}, confirming that \(T\) maps the standard Brownian motion to the GBM.\\
\\
Alternatively, \(T\) can be decomposed as \(T = T_2 \circ T_1\) (see Definition~\ref{group structure}), where:
\begin{itemize}
    \item \(T_1 = (\mathrm{Id}, I, 1, h)\): A pure Girsanov transformation (drift change).
    \item \(T_2 = (\Phi, I, 1, 0)\): A pure change of variable (Itô transformation).
\end{itemize}
 By Theorem~\ref{trasformate}, \(T_1\) maps \((X,W)\) to \((Z, \bar{W})\) where \(Z_t = X_t\) and \(\bar{W}_t\) is as defined above. The intermediate SDE on \((\Omega, \mathcal{F}, \mathbb{Q})\) has coefficients:
\[
\mu_1 = E_{T_1}(\mu_0) = h = \frac{\mu}{\sigma}-\frac{\sigma}{2}, \qquad \sigma_1 = E_{T_1}(\sigma_0) = \sigma_0 = 1.
\]
 Applying \(T_2\) to the result of \(T_1\), we transform \(Z_t\) into \(Y_t = \Phi(Z_t) = e^{\sigma X_t}\), while the Brownian motion \(\bar{W}\) remains unchanged (since \(h_2=0\)). The final coefficients are:
\begin{align*}
\mu_2 &= E_{T_2}(\mu_1) = \left(\mu_1 D\Phi + \tfrac{1}{2} D^2\Phi\right)\circ \Phi^{-1} 
= \left( \left(\tfrac{\mu}{\sigma}-\tfrac{\sigma}{2}\right)\sigma \Phi + \tfrac{1}{2}\sigma^2 \Phi \right)\circ \Phi^{-1} = \mu x, \\
\sigma_2 &= E_{T_2}(\sigma_1) = (D\Phi)\circ \Phi^{-1} = \sigma x.
\end{align*}
This confirms that \(E_{T_2 \circ T_1} = E_{T_2} \circ E_{T_1}\) and \(P_{T_2 \circ T_1} = P_{T_2} \circ P_{T_1}\), yielding the same GBM solution \eqref{GBM1}. The process is summarized in the following diagram:

\[
\begin{tikzcd}
{} \arrow[dd, "T=T_2\circ T_1"', bend right=50] & dX_t=dW_t \quad \text{in} \quad (\Omega, \mathcal{F}, (\mathcal{F}_t)_t,\mathbb{P}) \arrow[d, "T_1"] \\
{} & dZ_t=(\frac{\mu}{\sigma}-\frac{\sigma}{2})dt+d\bar{W_t} \quad \text{in} \quad  (\Omega, \mathcal{F}, (\mathcal{F}_t)_t,\mathbb{Q}) \arrow[d, "T_2"] \\
{} & dY_t = \mu Y_t dt + \sigma Y_t d\bar{W}_t \quad \text{in} \quad  (\Omega, \mathcal{F}, (\mathcal{F}_t)_t,\mathbb{Q})
\end{tikzcd}
\]
\subsection{Lie algebra structure and infinitesimal transformations}

The group of stochastic transformations lacks an obvious intrinsic differentiable structure. To define a Lie algebra, we utilize a geometric approach: we identify the stochastic transformations with a closed subgroup of the diffeomorphism group of a specific principal bundle. This allows the stochastic group to inherit a natural Lie group structure from the geometry of the bundle.

Consider the \emph{trivial principal bundle} \( P = \tilde{M} \times G \). Here, the base manifold \( \tilde{M}:=M\times \mathbb{R}_+ \) represents the spacetime domain (with coordinates \( (x,t) \)), while the \emph{structure group} is defined as \( G = \mathrm{SO}(d) \times \mathbb{R}_+ \times \mathbb{R}^n \). Conceptually, given a stochastic transformation \( T=(\Phi,f,B,h)\), the fiber $ G $ encodes the parameters $(B, f', h)$ that act on the driving Brownian motion, while the base $ \tilde{M} $ correspond to the components $ (\Phi, f) $ that act on the solution process $ X $. The bundle is equipped with the standard projection \(\pi_M: ((x,t),g) \mapsto (x,t)\) and the right action \(R_{g_1}((x,t), g) = ((x,t), g * g_1)\). Here \(*\) denotes the group operation in \(G\) , which is not fixed a priori but must be determined to ensure compatibility with the composition of stochastic transformations.

\begin{definizione}\label{isomorfismo fibrati}  Given two (trivial) principal bundles $ \tilde{M} \times G$ and $  \tilde{M'} \times  G $, an isomorphism between principal bundles is a diffeomorphism $ F:  \tilde{M} \times G \rightarrow \tilde{M'} \times G$ that preserves the principal bundle structure of $ \tilde{M} \times G$ and $ \tilde{M'} \times G$. In other words, there exists a diffeomorphism $\Psi: \tilde{M} \rightarrow \tilde{M'}$ such that 
\begin{equation*}
 \pi_{\tilde{M'}} \circ F = \Psi \circ \pi_{\tilde{M}} \qquad  F \circ R_{g} = R_{g} \circ F \quad \forall g\in G.
\end{equation*} 
\end{definizione}

\begin{osservazione}\label{identificazione}
A bundle isomorphism \(F\) is uniquely determined by its value at the identity section \(((x,t), e)\), where \(e=(\underline{0},1,I_d)\) is the neutral element in \( G\). This means that if \(F((x,t), e) = (\Psi(x,t), \tilde{g}(x,t))\), then \(F((x,t), h) = (\Psi(x,t), \tilde{g}(x,t) * h)\) for all \( h\) in \( G\).
\end{osservazione}

We define a map from the group of stochastic transformations to the group of bundle isomorphisms of \(\tilde{M} \times G\):
\begin{equation}\label{F_T}
T = (\Phi, f, B, h) \longmapsto F_T,
\end{equation}
where, thanks to Remark \ref{identificazione}, \(F_T\) is defined via its action on the identity section:
\[
F_T((x,t), e) := \big( \underbrace{(\Phi(x,t), f(t))}_{\Psi(x,t)},\, (B(x,t), f'(x,t), h(x,t)) \big).
\]
This map is well-defined and bijective: every \(T\) defines a unique \(F_T\), and conversely, any isomorphism \(F\) defines a unique \(T\) by extracting the base diffeomorphism \((\Phi, f)\) and the group components \((B, h)\) from \(F((x,t), e)\).

Since the group of bundle isomorphisms is a closed subgroup of \(\mathrm{Diff}(\tilde M \times G)\) (that is, the group of Diffeomorphisms of \( \tilde{M} \times G\)), it is a Lie group by Cartan's theorem. To transfer this structure to the group of stochastic transformations, \eqref{F_T} must be a group homomorphism. Specifically, we require \(F_{T_2 \circ T_1} = F_{T_2} \circ F_{T_1}\). This condition determines the group law \(*\) on \(G\).
Recalling the composition law for stochastic transformations (Definition~\ref{group structure}), we have
\begin{equation}\label{FT2circT1}
F_{T_2 \circ T_1}((x,t),e) = \Bigl((\Phi_2 \circ \Phi_1, f_2\circ f_1), \; (B_2 \circ \Phi_1) B_1,\; (f'_2 \circ \Phi_1)f'_1,\; \sqrt{f'_1} B_1^{-1}(h_2 \circ \Phi_1) + h_1\Bigr).
\end{equation}
 Conversely, evaluating the composition of isomorphisms \(F_{T_2}(F_{T_1}((x,t), e))\) yields:
\begin{equation}\label{FT2FT1}
F_{T_2}\big( (\Phi_1, f_1), (B_1, f'_1, h_1) \big) = \big( (\Phi_2 \circ \Phi_1, f_2 \circ f_1),\, (B_2 \circ \Phi_1, f'_2 \circ f_1, h_2 \circ \Phi_1) * (B_1, f'_1, h_1) \big).
\end{equation}
Comparing the group components of \eqref{FT2circT1} and \eqref{FT2FT1}, the operation \(*\) must satisfy:
\[
(B_2, f'_2, h_2) * (B_1, f'_1, h_1) = (B_2 B_1,\, f'_2 f'_1,\, \sqrt{f'_1} B_1^{-1} h_2 + h_1).
\]
This structure identifies \(G\) as a semidirect product. Let \(K = SO(d) \times \mathbb{R}_+\) with the component-wise product. We define the homomorphism \(\psi: K \to \mathrm{Aut}(\mathbb{R}^d)\) by \(\psi_{(B, f')}(h) = \sqrt{f'} B^{-1} h\), where \( \mathrm{Aut}(\mathbb{R}^d)\) denotes the group of automorphism of \( \mathbb{R}^d\). Then:
\[
G \cong K \ltimes_{\psi} \mathbb{R}^d
\]
with product law \(\ltimes_{\psi}\) defined in a standard way as
\[
(k_2, h_2) \ltimes_{\psi} (k_1, h_1) = (k_2 k_1, \psi_{k_1}(h_2)+h_1)=\bigl(k_2 k_1,\, \sqrt{f'_1} B_1^{-1} h_2 + h_1\bigr).
\]
With this group law, the map \(T \mapsto F_T\) is a group isomorphism, preserving inverses (\(F_{T^{-1}} = F_T^{-1}\)) and the identity (\( F_{T_0}=Id\)).

\begin{teorema}
The group of stochastic transformations is isomorphic to the group of bundle isomorphisms of \(\tilde{M} \times G\), where \( \tilde{M}=M\times \mathbb{R}_+\) and \(G = (SO(d) \times \mathbb{R}_+) \ltimes_{\psi} \mathbb{R}^d\). Consequently, the group of stochastic transformations inherits a Lie group structure.
\end{teorema}

We can now define the Lie algebra \(V_d(M)\) associated with this group. Elements of the Lie algebra correspond to tangent vectors at the identity of a one-parameter subgroup \(\{T_\lambda\}_{\lambda}\), computed as derivatives at \(\lambda=0\):
\begin{equation}\label{algebra_di_lie}
\begin{aligned}
Y(x,t) &= \partial_\lambda \Phi_\lambda(x,t)\big|_{\lambda=0}, & \quad m(t) &= \partial_\lambda f_\lambda(t)\big|_{\lambda=0}, \\
C(x,t) &= \partial_\lambda B_\lambda(x,t)\big|_{\lambda=0}, & \quad H(x,t) &= \partial_\lambda h_\lambda(x,t)\big|_{\lambda=0}.
\end{aligned}
\end{equation}

\begin{definizione}\label{def ininitesimal transformation}
An \emph{infinitesimal stochastic transformation} is an element \(V = (Y, m, C, H) \in V_d(M)\), where \(Y=\sum\limits_{i=1}^n Y_i(x,t)\partial_{x_i}\) is a vector field on \(M\), \( m=m(t)\) is a function on \(\mathbb{R}_+\), \(C: M \times \mathbb{R}_+ \to \mathfrak{so}(d)\) is an antysimmetric matrix, and \(H: M \times \mathbb{R}_+\to \mathbb{R}^d\). If \(V = (Y, 0, 0, 0)\), it is called a \emph{strong} infinitesimal transformation.
\end{definizione}

Finally, given an element of the Lie algebra, we can recover the corresponding finite transformation.

\begin{teorema}[Reconstruction of the flow]\label{teo ricostruzione del flusso}
Let \(V = (Y, C, \tau, H)\) be an infinitesimal stochastic transformation. The corresponding one-parameter group \(T_\lambda\) is determined by the unique solution to the system:
\begin{equation}\label{ricostruzione del flusso}
\begin{cases}
    \partial_{\lambda}\Phi_{\lambda}(x,t) = Y(\Phi_{\lambda}(x,t),t), & \Phi_0 = \mathrm{Id}_{\mathbb{R}^n}, \\
    \partial_{\lambda} f_{\lambda}(t) = m(f_{\lambda}(t)), & f_0 = \mathrm{id}_{\mathbb{R}}, \\
    \partial_{\lambda}B_{\lambda}(x,t) = C(\Phi_{\lambda}(x,t),t) B_{\lambda}(x,t), & B_0 = I_d, \\
    \partial_{\lambda}h_{\lambda}(x,t) = \sqrt{f'_{\lambda}(t)}B_{\lambda}^{-1}(x,t) H(\Phi_{\lambda}(x,t),t), & h_0 = 0.
\end{cases}
\end{equation}
\end{teorema}
\begin{proof}
See \cite{paper2020}.
\end{proof}

\begin{osservazione}
    Following the usual one-parameter group notations, we will denote by \( T_{-\lambda}\) the inverse of the transformation \( T_\lambda\), that is, \( T_{-\lambda}=T_\lambda^{-1}\). In the same way, \( \Phi_{-\lambda}\equiv \Phi_\lambda^{-1}\), \( f_{-\lambda}\equiv f_\lambda^{-1}\), \( B_{-\lambda}\equiv B_\lambda^{-1}\) and \( h_{-\lambda}=h_\lambda^{-1}.\)
\end{osservazione}

\subsubsection{Time invariance of Brownian motion}\label{Example time invariance BM}

Consider a \( n-\)dimensional Brownian motion \( d\underline{X}_t=d\underline{W}_t\), with \( \underline{X}=(X^1,...,X^n)^T, \underline{W}=(W^1,...,W^n)^T,\) drift \( \mu_0:=\underline{0}\) and diffusion coefficient \( \sigma_0:=I_n\). Let us apply to this SDE the infinitesimal stochastic transformation \( V=(Y,m, C, H)\), defined by:
\[ 
Y = \sum_{i=1}^n x_i \partial_{x_i}, \quad m(t)= 2t, \quad C = \underline{0}_{n\times n},  \quad H = \underline{0}_n.
\]
We aim to recover the corresponding one-parameter group of finite transformations. By the flow reconstruction theorem (Theorem \ref{teo ricostruzione del flusso}), the components of \(T_\lambda\) satisfy:
\[ 
\partial_\lambda \Phi_\lambda(x,t) = \Phi_\lambda(x,t), \quad \partial_\lambda f_\lambda(t) = 2f_\lambda(t), \quad \partial_\lambda B_\lambda = \underline{0}, \quad \partial_\lambda h_\lambda = 0.
\]
The last two equations imply that \( B_\lambda\) and \( h_\lambda\) are independent of \( \lambda\). The first two equations can be solved by separation of variables:
\[ 
\frac{d\Phi_\lambda}{\Phi_\lambda} = d\lambda, \qquad \frac{df_\lambda}{f_\lambda} = 2d\lambda.
\]
Imposing the initial conditions \( \Phi_0=\mathrm{Id}, f_0=\mathrm{id}, B_0=I_{n}, h_0=\underline{0}_n\), we obtain the group \( T_\lambda=(\Phi_\lambda, f_\lambda, B_\lambda, h_\lambda) \):
\[ 
\Phi_\lambda(x) = e^{\lambda}x, \quad f_\lambda(t) = e^{2\lambda} t, \quad B_\lambda = I_n, \quad h_\lambda = \underline{0}_n
\]
with \( \lambda\) ranging in \(\mathbb{R}\).
Setting \( a = e^\lambda \) (with \( a > 0 \)), the transformation \( T := T_{\ln(a)} \) takes the form:
\[
\Phi(x) = a x, \quad f(t) = a^2 t, \quad B = I_n, \quad h = \underline{0}_n.
\]
To verify the action of \( T\) on the original Brownian motion SDE, we apply Theorem \ref{trasformate}. Since \( f'(t) = \partial_t f(t) \) and \( f^{-1}(t)=\frac{t}{a^2}\), the transformed process is given by the pair:
\[ 
P_T(X)_t = \Phi(X_{f^{-1}(t)}) = a X_{t/a^2}, \qquad 
P_T(W)_t = \int_0^{f^{-1}(t)}\sqrt{f'(s)}\,dW_s = \int_0^{t/a^2} a\,dW_s = a W_{t/a^2}.
\]
On the other hand, the transformed coefficients are 
\begin{align*}
E_T(\mu_0) &= \Big[\frac{1}{f'}\big(L_t(\Phi)+D(\Phi)\sigma h\big)\Big]\circ(\Phi^{-1},f^{-1}) = \Big[\frac{1}{a^2}\big(0 + 0\big)\Big] = \underline{0}_n=\mu_0, \\
E_T(\sigma_0) &= \Big(\frac{1}{\sqrt{f'}}D(\Phi)\sigma B^{-1}\Big)\circ(\Phi^{-1},f^{-1}) = \Big(\frac{1}{a}\, a I_n \cdot I_n\Big) = I_n=\sigma_0.
\end{align*}
Since \( h=\underline{0}_n \), the probability measure remains unchanged (\( d\mathbb{Q}/d\mathbb{P}=1\)). However, the filtration is time-scaled (Theorem \ref{trasformate}):
\[ 
\mathcal{F}'_{t} = \mathcal{F}_{f^{-1}(t)} = \mathcal{F}_{t/a^2}.
\]
Since \( E_T(\mu_0)=\mu_0\) and \( E_T(\sigma_0)=\sigma_0\), \( T\) is a weak symmetry according to Definition \ref{def symmetries}. The transformed process \( P_T(X,W)\) satisfies the original Brownian motion SDE with respect to the transformed filtration \( (\mathcal{F}_{t/a^2})_t \). This recovers the well-known self-similarity (or time invariance) property of Brownian motion (see Proposition 3.2 in \cite{baldi}): if \( W_t\) is a \( \mathcal{F}_t-\)Brownian motion, then \( a W_{\frac{t}{a^2}}\) is a \( \mathcal{F}_{\frac{t}{a^2}}-\)Brownian motion.

 \subsection{Infinitesimal symmetries}

\begin{definizione}
A stochastic infinitesimal transformation \(V\) generating a one-parameter group \(T_{\lambda}\) is called an infinitesimal symmetry---respectively, strong, weak, or \(\mathcal{G}\)-weak---for  \(SDE_{\mu,\sigma}\) if \(T_{\lambda}\) is a finite symmetry of the corresponding type for \(SDE_{\mu,\sigma}\).
\end{definizione}

The next theorem provides a useful characterization of infinitesimal symmetries, yielding the \emph{infinitesimal determining equations}. These constitute the infinitesimal counterparts of the equations in Theorem~\ref{equazioni determinanti finite}, but are far simpler to solve and therefore form an effective tool for the explicit computation of symmetries.

\begin{teorema}\label{equazioni determinanti}
Let \(V = (Y,0,0,0)\) be a strong infinitesimal transformation. Then \(V\) is a strong infinitesimal symmetry if and only if \(V\) generates a one-parameter group of transformations and the following conditions hold:
\[
Y(\mu) - L_t(Y) = 0, 
\qquad 
[Y,\sigma] = 0,
\]
where
\[
[Y,\sigma]_{i \alpha}
=
Y(\sigma_{i\alpha})
-
\sum\limits_{k=1}^n\partial_{k}(Y^{i})\,\sigma^{k}_{\alpha}
=
\sum\limits_{k=1}^n\big(Y^{k}\partial_{k}(\sigma^{i}_{\alpha})
-
\partial_{k}(Y^{i})\,\sigma^{k}_{\alpha}\big).
\]

Let \(V = (Y, m, C, H)\) be a weak infinitesimal transformation. Then \(V\) is a weak infinitesimal symmetry of \(SDE_{\mu,\sigma}\) if and only if \(V\) generates a one-parameter group of transformations and the following conditions hold:
\begin{equation}\label{weak inf}
Y(\mu)+m(t)\partial_t \mu - L_t(Y) - \sigma H + m'(t) \mu = 0,
\qquad
[Y,\sigma] + \tfrac{1}{2}m'(t) \sigma + \sigma C = 0.
\end{equation}

Let \(V = (Y,m,C,H)\) be a weak infinitesimal transformation. Then \(V\) is a \(\mathcal{G}\)-weak infinitesimal symmetry of \(SDE_{\mu,\sigma}\) if and only if \(V\) generates a one-parameter group of transformations  and the following determining equations hold:
\begin{equation}\label{g weak inf}
Y(\mu) + m(t)\partial_t \mu - L_t(Y) - \sigma H + m'(t) \mu = 0,
\qquad
[Y,\sigma\sigma^{T}] + \tau\, \sigma\sigma^{T} = 0,
\end{equation}
where
\[
[Y,\sigma\sigma^{T}]
=
Y(\sigma\sigma^{T})
-
D(Y)\,\sigma\sigma^{T}
-
\sigma\sigma^{T}\, D(Y)^{T}.
\]
\end{teorema}

\begin{proof}
See \cite{papernuovo}. The main idea consists of differentiating the corresponding finite determining equations with respect to the flow parameter.
\end{proof}

Now that we have established a Lie algebra structure, our strategy to study the invariance properties of a given diffusion is the following: instead of solving the finite determining equations given in Theorem \ref{equazioni determinanti finite}, which are generally difficult to handle, we first solve the infinitesimal determining equations provided in Theorem \ref{equazioni determinanti}, and then apply Theorem \ref{teo ricostruzione del flusso} to recover the corresponding finite symmetries.

\subsection{Symmetries of Brownian Motion}\label{Symmetries of BM}

We aim to recover the symmetries of an $n$-dimensional Brownian motion described by the SDE $d\mathbf{X}_t = d\mathbf{W}_t$, where $\mathbf{X} = (X^1, \dots, X^n)^T$ and $\mathbf{W} = (W^1, \dots, W^n)^T$, with drift $\mu = \mathbf{0}$ and diffusion coefficient $\sigma = I_n$.

Our approach proceeds in two steps: first, we solve the infinitesimal determining equations to identify the infinitesimal symmetries; second, we reconstruct the associated flow to derive the corresponding group of finite symmetries.

According to Theorem \ref{equazioni determinanti}, a weak infinitesimal symmetry is defined by the quadruple $V = (Y, m, C, H)$, where:
\[
Y = \sum_{i=1}^n Y^{i}(\mathbf{x}, t) \partial_{x_i}, \quad m = m(t), \quad C = [c_{ij}(\mathbf{x},t)]_{i,j=1,\dots,n}, \quad H = \begin{pmatrix} H^1(\mathbf{x}, t) \\ \vdots \\ H^n(\mathbf{x}, t) \end{pmatrix}.
\]
The infinitesimal generator must satisfy the determining equations \eqref{weak inf}. Given $\mu = \underline{0}$ and $\sigma = I_n$, the first determining equation, $L_t(Y^{i}) - Y(\mu^{i})-m(t)\partial_t(\mu^i) + (\sigma H)^{i} - m'(t) \mu^{i} = 0$, reduces to:
\begin{equation}\label{eqdet1 mb}
H^{i}(\mathbf{x}, t) = -\partial_t Y^{i}(\mathbf{x}, t) - \frac{1}{2} \Delta Y^{i}(\mathbf{x}, t).
\end{equation}
Since for every \(k,i,\alpha\) we have \( \partial_k \sigma_{i \alpha}=0\), the second determining equation, $[Y, \sigma] + \frac{1}{2}m'(t) \sigma + \sigma C = 0$, simplifies to $-D(Y) + \frac{1}{2}m'(t) I_n + C = 0$. In component form, this is expressed as:
\begin{equation}\label{eqdet2 mb}
- \frac{\partial Y^i}{\partial x_j} + \frac{1}{2} m'(t) \delta_{ij} + c_{ij} = 0.
\end{equation}
Combining \eqref{eqdet1 mb} and \eqref{eqdet2 mb}, we obtain the complete system of determining equations:
\begin{equation}\label{eqdet tot mb}
\begin{cases}
H^{i}(\mathbf{x}, t) = -\partial_t Y^{i}(\mathbf{x}, t) - \dfrac{1}{2} \Delta Y^{i}(\mathbf{x}, t), & \forall i = 1, \dots, n, \\[0.8em]
\partial_{x_i} Y^{i}(\mathbf{x}, t) = \dfrac{1}{2} m'(t), & \forall i = 1, \dots, n, \\[0.8em]
\partial_{x_j} Y^{i}(\mathbf{x}, t) = c_{ij}(\mathbf{x},t), & \forall i \neq j.
\end{cases}
\end{equation}
This system depends on two sets of free time-dependent parameters, $m(t)$ and $C(t)$, which generate two independent families of infinitesimal symmetries.

\paragraph{First Family ($V_\alpha$)}
By setting $C = 0$ and $m = \alpha(t)$, where $\alpha(t)$ is an arbitrary smooth function, we obtain the infinitesimal symmetry:

\begin{equation}\label{Valpha nd}
V_{\alpha} = \left(  \sum\limits_{i=1}^n\frac{1}{2}\alpha'(t) x_i \partial_{x_i} , \ \alpha(t),  \ \mathbf{0}, \  - \frac{1}{2} \alpha''(t) \mathbf x \right).
\end{equation}

\paragraph{Second Family ($V_\beta$)}
If we set $m = 0$, the determining system \eqref{eqdet tot mb} implies that the spatial part of the symmetry is governed by the antisymmetric matrix $C(t) = (c_{ij}(t)) \in \mathfrak{so}(n)$. Solving the equations\footnote{Recall that since \( C\) is by construction antisymmetric then \(c_{ij}=-c_{ji} \forall i\not=j\) and \( c_{ii}=0 \forall i\). }, we find:

\begin{equation}\label{Vbeta nd}
V_{\beta} = \left( \Big(\sum\limits_{i,j=1}^n C_{\beta(t)_{i,j}} x_j\partial_{x_i}\Big) , \ 0, \ C_{\beta(t)}, -C'_{\beta(t)} \mathbf{x} \right),
\end{equation}
where $C_{\beta(t)}$ is a smooth antisymmetric matrix parametrized by the vector of functions $\beta(t)$. 

In the one-dimensional case ($n=1$), rotations are trivial ($C=0$), and the spatial part $Y$ becomes an arbitrary function of time $\beta(t)$. Thus, $V_\beta$ reduces to:

\begin{equation}\label{Vbeta 1d}
V_{\beta} = \left(  \beta(t) \partial_x,\ 0, \ 0, -\beta'(t) \right).
\end{equation}
In the two-dimensional case ($n=2$), $V_\beta$ explicitly accounts for rotations:

\begin{equation}\label{Vbeta 2d}
V_{\beta} = \left(  \beta(t)(x_2 \partial_{x_1} - x_1 \partial_{x_2}), \ 0, \ \begin{pmatrix} 0 & \beta(t) \\ -\beta(t) & 0 \end{pmatrix}, \begin{pmatrix} -x_2\beta'(t) \\ x_1\beta'(t) \end{pmatrix} \right).
\end{equation}
We now proceed to reconstruct the flow and recover the corresponding finite symmetries. While the study of $T_\lambda$ for $V_\alpha$ is conducted in arbitrary dimension $n$, for $V_\beta$ we focus on $n=1$ and $n=2$ to illustrate the distinction between translational and rotational invariance. The results for $n \geq 3$ generalize directly from the two-dimensional case.

\subsubsection*{The family $V_{\alpha}$ and time invariance}

Consider an $n$-dimensional Brownian motion $d\mathbf{X}_t = d\mathbf{W}_t$, where $\mathbf{X} = (X^1, \dots, X^n)^T$ and $\mathbf{W} = (W^1, \dots, W^n)^T$. Consider the family of infinitesimal symmetries $V_{\alpha}$ computed in \eqref{Valpha nd}
\[
V_{\alpha} = \left(\sum\limits_{i=1}^n\frac{1}{2} \alpha'(t) x_i\partial_{x_i}, \ \alpha(t), \ 0, \ -\frac{1}{2} \alpha''(t) \mathbf{x} \right).
\]
This family is related to the time-invariance properties of Brownian motion. It is well known (see \cite{oksendal}) that if $\mathbf{W}_t$ is a Brownian motion on a filtered probability space $(\Omega, \mathcal{F}, (\mathcal{F}_t), \mathbb{P})$ and $f(t) = \int_0^t f'(s) ds$ is a time change, then the process $\bar{\mathbf{W}}_t := \int_0^t \sqrt{f'(f^{-1}(s))} d\mathbf{W}_s$ is a $(\Omega, \mathcal{F}, (\mathcal{F}_{f^{-1}(t)}), \mathbb{P})$-Brownian motion. The family $V_\alpha$ generalizes this by considering, instead of \(\bar{W}\), the process $\tilde{\mathbf{W}}_t := \sqrt{f'(f^{-1}(t))} \mathbf{W}_{f^{-1}(t)}$. While $\tilde{\mathbf{W}}$ is not a $\mathbb{P}$-Brownian motion, Itô’s calculus and Girsanov's theorem show it is a $(\Omega, \mathcal{F}, (\mathcal{F}_{f^{-1}(t)}), \mathbb{Q})$-Brownian motion, where $d\mathbb{Q}/d\mathbb{P}$ is the Doléans-Dade exponential of the drift $h := \frac{f''(f^{-1}(t))}{2 [f'(f^{-1}(t))]^{3/2}}$, since:
\[
d\tilde{\mathbf{W}}_t = \underbrace{\sqrt{f'(f^{-1}(t))} \, d\mathbf{W}_{f^{-1}(t)}}_{d\bar{\mathbf{W}}_t} - \underbrace{\frac{f''(f^{-1}(t))}{2 [f'(f^{-1}(t))]^{3/2}}  \mathbf{W}_{f^{-1}(t)}  dt}_{\mathbf{h} \,  dt}.
\]
According to the flow reconstruction (Theorem \ref{teo ricostruzione del flusso}), $V_\alpha$ generates a one-parameter group of symmetries $T_\lambda = (\Phi_\lambda, f_\lambda, B_\lambda,  h_\lambda)$ satisfying:
\[
\partial_\lambda f_\lambda(t) = \alpha(f_\lambda(t)), \quad \partial_\lambda \Phi_\lambda = \frac{1}{2} \alpha'(f_\lambda(t)) \Phi_\lambda, \quad \partial_\lambda B_\lambda=0, \quad \partial_\lambda h_\lambda = -\frac{1}{2} \sqrt{f'_\lambda} \alpha''(f_\lambda(t)) \Phi_\lambda.
\]
Integrating with respect to $\lambda$ with initial conditions $\Phi_0(\mathbf{x}, t) = \mathbf{x}$, $f_0(t) = t$, \( B_0(\mathbf{x},t)=I_n\) and $h_0(\mathbf{x},t) = \mathbf{0}$, and applying the chain rule $\alpha'(f_\lambda(t)) = \frac{\partial_\lambda \alpha(f_\lambda(t))}{\alpha(f_\lambda(t))}$, we obtain:
\begin{align*}
\Phi_\lambda(\mathbf{x},t) &= \mathbf{x} \sqrt{|\alpha(f_\lambda(t))|}, \quad f'_\lambda(t) = |\alpha(f_\lambda(t))|, \quad B_\lambda=I_n \\
\mathbf{h}_\lambda(\mathbf{x},t) &= -\frac{1}{2} \mathbf{x} \int \alpha''(f_\lambda(t)) |\alpha(f_\lambda(t))| d\lambda = -\frac{1}{2} \mathbf{x} \, \mathrm{sgn}(\alpha(f_\lambda(t))) \alpha'(f_\lambda(t)).
\end{align*}
By Theorem \ref{trasformate}, the transformed pair $(P_{T_\lambda}(\mathbf{X}), P_{T_\lambda}(\mathbf{W}))$ solves the same SDE as $(\mathbf{X}, \mathbf{W})$. Specifically, the transformed processes are:
\[ 
P_{T_\lambda}(\mathbf{X}_t) = \Phi_\lambda(X_{f_{-\lambda}(t)}, f_{-\lambda}(t))=\sqrt{|\alpha(t)|} \mathbf{X}_{f_{-\lambda}(t)},\]
\[\quad P_{T_\lambda}(\mathbf{W}_t) = \int_0^{f_{-\lambda}(t)}\sqrt{f'_\lambda(t)}\big(d\mathbf{W}_t-h(\mathbf{X}_s,s)ds\big)=\int_0^t \sqrt{|\alpha(s)|} \left( d\mathbf{W}_{f_{-\lambda}(s)} + \frac{1}{2} \mathbf{X}_{f_{-\lambda}(s)} \frac{\mathrm{sgn}(\alpha(s)) \alpha'(s)}{|\alpha(s)|} ds \right).
\]
Consistent with the original equation $d\mathbf{X}_t = d\mathbf{W}_t$, Itô's calculus and Girsanov's theorem confirm that $P_{T_\lambda}(\mathbf{X}, \mathbf{W})$ solves the SDE under the measure $\mathbb{Q}_\lambda$, since:
\[
d \left( \sqrt{|\alpha(t)|} \mathbf{X}_{f_{-\lambda}(t)} \right) = \underbrace{\sqrt{|\alpha(t)|} d\mathbf{W}_{f_{-\lambda}(t)}}_{d\bar{\mathbf{W}}_t} - \left( \underbrace{- \frac{1}{2\sqrt{|\alpha(t)|}} \mathbf{X}_{f_{-\lambda}(t)} \mathrm{sgn}(\alpha(t)) \alpha'(t)}_{\tilde{h}} dt \right).
\]

\subsubsection*{One-dimensional family \(V_\beta\) and Girsanov invariance}

Consider a one-dimensional Brownian motion \(dX_t = dW_t\) and the family of infinitesimal symmetries \(V_\beta\) computed in \eqref{Vbeta 1d}:
\[ V_\beta = \left( \beta(t)\partial_x,\  0, \ 0,\ -\beta'(t) \right). \]
The vector field \(V_\beta\) encodes the invariance properties of Brownian motion described by Girsanov's theorem. Specifically, if \(W\) is a \(\mathbb{P}\)-Brownian motion, then \(W_t - \int_0^t h(X_s, s) ds\) is a \(\mathbb{Q}\)-Brownian motion under the measure change defined by the Doléans-Dade exponential of \(h\).

By reconstructing the flow (Theorem \ref{teo ricostruzione del flusso}), \(V_\beta\) generates a one-parameter group of finite symmetries \(T_\lambda = (\Phi_\lambda, f_\lambda, B_\lambda, h_\lambda)\) satisfying:
\[ \frac{d\Phi_\lambda}{d\lambda} = \beta(t), \quad \frac{df_\lambda}{d\lambda} = 0, \quad \frac{dB_\lambda}{d\lambda} = 0,  \quad \frac{dh_\lambda}{d\lambda} = -\beta'(f_\lambda(t)), \]
with identity initial conditions at \(\lambda=0\). The resulting finite symmetry \(T_\lambda\) consists of the spatial diffeomorphism \(\Phi_\lambda(x,t) = x + \lambda\beta(t)\) and a Girsanov drift transformation \(h_\lambda(x,t) = -\lambda\beta'(t)\).

Since \(V_\beta\) is a symmetry, the transformed process satisfies the original SDE under the new measure \(\mathbb{Q}_\lambda\). According to Theorem \ref{trasformate}, the transformed processes are:
\[ P_{T_\lambda}(X_t) = \Phi_\lambda(X_t, t) = X_t + \lambda\beta(t), \]
\[ P_{T_\lambda}(W_t) = W_t - \int_0^t h_\lambda(X_s, s) ds = W_t + \lambda\int_0^t \beta'(s) ds = W_t + \lambda\beta(t). \]
Consistent with Itô calculus and the SDE \(dX_t = dW_t\), it follows that \(d(X_t + \lambda\beta(t)) = d(W_t + \lambda\beta(t))\), which is a \(\mathbb{Q}_\lambda\)-Brownian motion by Girsanov's theorem.

\subsubsection*{Two-dimensional family \(V_\beta\) and rotational invariance}

Consider a two-dimensional Brownian motion \(d\mathbf{X}_t = d\mathbf{W}_t\), where \(\mathbf{X} = (X^1, X^2)^T\) and \(\mathbf{W} = (W^1, W^2)^T\). Consider the two-dimensional family of infinitesimal symmetries \(V_\beta\) computed in \eqref{Vbeta 2d}
%\[ V_{\beta} = \left(\begin{pmatrix}( \beta(t)(x^2\partial_{x^1} - %x^1\partial_{x^2})\\0 \end{pmatrix},
\[ V_{\beta} = \left(\beta(t)(x_2\partial_{x_1} - x_1\partial_{x_2}),\ 0,
\begin{pmatrix} 0 & \beta(t) \\ -\beta(t) & 0 \end{pmatrix},
\begin{pmatrix} -x_2\beta'(t) \\ x_1\beta'(t) \end{pmatrix} \right). \]
This family generalizes the well-known rotational invariance of Brownian motion. By Lévy's characterization, if \(\mathbf{W}_t\) is a \(\mathbb{P}\)-Brownian motion and \(B(t)\) is a rotation matrix, then \(\bar{\mathbf{W}}_t = \int_0^t B(s) d\mathbf{W}_s\) is still a \(\mathbb{P}\)-Brownian motion. \(V_\beta\) extends this property by considering the directly rotated process \(\tilde{\mathbf{W}}_t := B(t)\mathbf{W}_t\). Although \(\tilde{\mathbf{W}}\) is not a Brownian motion under \(\mathbb{P}\), Girsanov's theorem ensures it is one under a new measure \(\mathbb{Q}\) with Radon-Nikodym derivative given by the Doléans-Dade exponential of \(h(\mathbf{x}, t) = -B'(t)\mathbf{x}\), where  \( B'(t)\)is the matrix whose entries are the time derivatives of the entries of \( B(t).\)

From the reconstruction of the flow (Theorem \ref{teo ricostruzione del flusso}), \(V_\beta\) generates a one-parameter group of finite symmetries \(T_\lambda = (\Phi_\lambda, f_\lambda, B_\lambda, h_\lambda)\) defined by:
\begin{align*}
\Phi_\lambda (\mathbf{X}_t, t) &= \begin{pmatrix} \cos(\lambda\beta(t))X^1_t + \sin(\lambda\beta(t))X^2_t \\ -\sin(\lambda\beta(t))X^1_t + \cos(\lambda\beta(t))X^2_t \end{pmatrix}, \quad f_\lambda(t) = t, \\
B_\lambda(t) &= \begin{pmatrix} \cos(\lambda\beta(t)) & \sin(\lambda\beta(t)) \\ -\sin(\lambda\beta(t)) & \cos(\lambda\beta(t)) \end{pmatrix}, \quad h_\lambda(\mathbf{X}_t, t) = \begin{pmatrix} -\lambda \beta'(t) X^2_t \\ \lambda \beta'(t) X^1_t \end{pmatrix}.
\end{align*}
Since \(V_\beta\) is a symmetry, the transformed process satisfies the original SDE under the measure \(\mathbb{Q}_\lambda\). According to Theorem \ref{trasformate}, the transformed processes are:
\[ P_{T_\lambda}(\mathbf{X}_t) = \Phi_\lambda(\mathbf{X}_t, t), \quad P_{T_\lambda}(\mathbf{W}_t) = \int_0^t B_\lambda(s) \left( d\mathbf{W}_s - h_\lambda(\mathbf{X}_s, s) ds \right). \]
Recalling the original equation \(d\mathbf{X}_t = d\mathbf{W}_t\), it is straightforward to verify that \( T_\lambda\) acts as a symmetry. Applying Itô's formula reveals that the transformed process acquires a drift term, which can be removed via Girsanov's theorem to recover the original Brownian dynamics under a new measure \(\mathbb{Q}_\lambda\). Explicitly, we have:
\footnotesize\begin{equation*}
d \begin{pmatrix} \cos(\lambda\beta(t))X^1_t + \sin(\lambda\beta(t))X^2_t \\ -\sin(\lambda\beta(t))X^1_t + \cos(\lambda\beta(t))X^2_t \end{pmatrix} = \underbrace{\begin{pmatrix} \cos(\lambda\beta(t)) & \sin(\lambda\beta(t)) \\ -\sin(\lambda\beta(t)) & \cos(\lambda\beta(t)) \end{pmatrix} \begin{pmatrix} dW^1_t \\ dW^2_t \end{pmatrix}}_{d\bar{\mathbf{W}}_t} - \underbrace{\begin{pmatrix} -\lambda\beta'(t)\sin(\lambda\beta(t))X^1_t + \lambda\beta'(t) \cos(\lambda\beta(t))X^2_t \\ -\lambda\beta'(t)\cos(\lambda\beta(t))X^1_t - \lambda\beta'(t)\sin(\lambda\beta(t))X^2_t \end{pmatrix}}_{\tilde{h}} dt.
\end{equation*}
\normalsize

\section{Lie symmetries and integration by parts formulas}\label{Lie symmetries and integration by parts formulas}  

In \cite{paper2023}, the study of invariance properties for diffusions in $\mathbb{R}^n$ via infinitesimal transformations led to the derivation of explicit, closed-form integration by parts formulas. This approach, inspired by the Bismut method in Malliavin calculus, was initially developed for symmetries that did not account for rotations. In \cite{papernuovo}, we extended this framework to include rotational transformations, proving the rotational invariance of the resulting integration by parts formula—a property deeply linked to the isotropy of Brownian motion. \\
\noindent
While the formula's structure remains invariant under rotations, the regularity assumptions required for its derivation undergo subtle changes. Furthermore, \cite{papernuovo} introduced the notion of $\mathcal{G}$-weak symmetry as the most general class of transformations for which this theorem holds, as it relies fundamentally on the invariance of the law of the solution process under the action of the symmetry group.\\
\\
\noindent
To ensure the validity of the theorem, we assume the following integrability condition:

\textbf{Hypothesis A.} \textit{Let $X_t$ be a solution to $SDE_{\mu, \sigma}$ with deterministic initial conditions. Each component of the following vectors or matrices:}
\begin{equation*}
    CH(X_t,t), \ H(X_t,t), \ Y(H)(X_t,t), \ L(Y)(X_t,t), \ \Sigma(Y)(X_t,t), \ L(Y(Y^i))(X_t,t), \ \Sigma(Y(Y^i))(X_t,t)
\end{equation*}
\textit{is in $L^2(\Omega)$ for $i=1, \dots, n$ and for all $t \in [0,T]$.  Here by \( \Sigma\) we mean the vector-valued differential operator defined as \( \Sigma=\sigma^T \cdot \nabla\), that is, 
\[ \Sigma_\alpha=\sum\limits_{j=1}^n\sigma_{j \alpha} \partial_{x_j}, \quad \alpha=1,..,d.\]}

\begin{teorema}[Integration by parts formula]\label{teo integrazione per parti 2}
Let $(X,W)$ be a solution to $SDE_{\mu,\sigma}$ and let $V=(Y, m,  C, H)$ be an infinitesimal stochastic symmetry for the system. Under Hypothesis A, the following identity holds for every $t \in [0,T]$ and for any bounded functional $F$ with bounded first and second derivatives:
\begin{equation}\label{integrazione per parti 2}
    -m(t) \mathbb{E}_{\mathbb{P}}[L(F(X_t))] + \mathbb{E}_{\mathbb{P}}\left[F(X_t) \int_0^t H(X_s,s) d W_s\right] + \mathbb{E}_{\mathbb{P}}[Y(F(X_t))] - \mathbb{E}_{\mathbb{P}}[Y(F(X_0))]=0.
\end{equation}
\end{teorema}

\begin{proof}
Let $T_{\lambda}=(\Phi_{\lambda}, f_\lambda, B_{\lambda}, h_{\lambda})$ be the one-parameter group generated by $V$. We provide here a sketch of the argument; the complete proof is available in \cite{papernuovo}.
\begin{itemize}
    \item Since $V$ is an infinitesimal symmetry, the generated flow $T_{\lambda}$ is a finite stochastic symmetry that preserves the law of the solution process. Thus, for any $t \in [0,T]$:
    \begin{equation*}
        \mathbb{E}_{\mathbb{P}}[F(X_t)] = \mathbb{E}_{\mathbb{Q_{\lambda}}}[F(P_{T_{\lambda}}(X_t))],
    \end{equation*}
    where $\mathbb{Q_{\lambda}}$ is the probability measure obtained via Girsanov's theorem under the transformation $T_{\lambda}$.
    
    \item Applying the Itô formula and utilizing the martingale property of the stochastic integral with respect to the $\mathbb{Q}_\lambda$-Brownian motion $P_{T_\lambda}(W)$, we obtain:
    \begin{equation*}
        \mathbb{E}_{\mathbb{P}}[F(X_t)] = \mathbb{E}_{\mathbb{Q_{\lambda}}} \left[\int_0^t L(F(\Phi_{\lambda}(X_{f_{-\lambda}(s)}))) ds\right]. 
    \end{equation*}
    
    \item Changing variables $s=f_\lambda(u)$ and applying the Radon-Nikodym derivative $\frac{d\mathbb{Q}_\lambda}{d\mathbb{P}}$ yields:
    \begin{equation*}
        \mathbb{E}_{\mathbb{P}}[F(X_t)] = \mathbb{E}_{\mathbb{P}}\left[\frac{d\mathbb{Q}_\lambda}{d\mathbb{P}}\bigg|_{\mathcal{F}_t} \int_0^{f_{-\lambda}(t)} L(F(\Phi_{\lambda}(X_s,s))) f'_{\lambda}(s) ds\right].
    \end{equation*}
    
    \item The result follows by differentiating both sides with respect to the flow parameter $\lambda$ and evaluating at $\lambda=0$. Since the left-hand side is independent of $\lambda$, its derivative vanishes. Applying the Leibniz rule on the right-hand side and moving the derivative inside the expectation (justified by the integrability conditions in Hypothesis A) concludes the proof.
\end{itemize}
\end{proof}

\begin{osservazione}
While the analytical conditions in Hypothesis A may appear stringent, they are often verifiable in practice using Lyapunov theory to ensure the required integrability. In the specific case of Brownian motion in $\mathbb{R}^n$, these conditions are naturally satisfied, as the moments of Gaussian random variables are finite for all orders. A detailed discussion on these regularity requirements is provided in \cite{paper2023, papernuovo}.
\end{osservazione}

\section{Integration by parts formulas and Stein identities}\label{Integration by parts and Stein identities}
We now apply our integration by parts theorem to the symmetries of Brownian motion derived in Section \ref{Symmetries of BM}. As we shall see, this result generalizes several well-known identities in probability theory, most notably the Stein identities. For the sake of conciseness, we focus on one-dimensional and two-dimensional Brownian motion—the latter being the first instance involving rotations—noting that the theorem extends naturally to general diffusions in $\mathbb{R}^n$. We omit the verification of the regularity assumptions, which can be found in \cite{papernuovo}.

Consider a one-dimensional Brownian motion $dX_t = dW_t$ in $\mathbb{R}$ and the family of symmetries $V_\beta$ computed in \eqref{Vbeta 1d}:
%\[ V_{\beta} = \left( \begin{pmatrix}\beta(t)\partial_x\\0\end{pmatrix}, \underline{0}, -\beta'(t) \right). \]
\[ V_{\beta} = \left( \beta(t)\partial_x, \ 0, \ 0, -\beta'(t) \right). \]
Theorem \ref{teo integrazione per parti 2} ensures that \textbf{for any time-dependent function $\beta(t)$} and any bounded functional $F \in C^2_b(\mathbb{R}^n)$, we have:
\begin{equation}\label{intVbeta 1d}
\mathbb{E}_{\mathbb{P}}\left[F(W_t)\int_0^t -\beta'(s) dW_s\right] + \mathbb{E}_{\mathbb{P}}[\beta(t)F'(W_t)] - \mathbb{E}_{\mathbb{P}}[\beta(0)F'(W_0)] = 0.
\end{equation}
Setting $\beta(t) = t$, identity \eqref{intVbeta 1d} reduces to:
\begin{equation}\label{stein identity 1d}
    \mathbb{E}_{\mathbb{P}}[W_t F(W_t)] = t \mathbb{E}_{\mathbb{P}}[F'(W_t)].
\end{equation}
Equation \eqref{stein identity 1d} is precisely the Stein identity for a Brownian motion at fixed time $t$. Recall that the well-known Stein's Lemma characterizes a random variable $Z \sim \mathcal{N}(\mu, \sigma^2)$ by the identity $\mathbb{E}[(Z-\mu)F(Z)] = \sigma^2 \mathbb{E}[F'(Z)]$; since $W_t \sim \mathcal{N}(0, t)$, the correspondence is exact.

The result in Theorem \ref{teo integrazione per parti 2} is closely related to the Bismut approach to Malliavin calculus. In one dimension, the family $V_\beta$ reflects the invariance of Brownian motion under a Girsanov-type change of measure. Specifically, considering a constant drift $h_\lambda = \lambda$, we know that if $ (W_t)_t$ is a $\mathbb{P}$-Brownian motion, then $ (W_t - \lambda t)_t
$ is a $\mathbb{Q}_\lambda$-Brownian motion with Radon-Nikodym derivative:
\[ \frac{d\mathbb{Q}_\lambda}{d\mathbb{P}}\bigg|_{\mathcal{F}_t} = e^{\lambda W_t - \frac{1}{2}\lambda^2 t}. \]
Thus, for any sufficiently smooth $F$:
\[ \mathbb{E}_{\mathbb{P}}[F(W_t)] = \mathbb{E}_{\mathbb{Q}_\lambda}[F(W_t - \lambda t)] = \mathbb{E}_{\mathbb{P}}\left[ e^{\lambda W_t - \frac{1}{2}\lambda^2 t} F(W_t - \lambda t) \right]. \]
If we now differentiate both sides with respect to $\lambda$ and evaluate the result at $\lambda = 0$, we obtain exactly the Stein identity for 1-dimensional Brownian motion \eqref{stein identity 1d}.

Stein's identity is not the unique classical result recoverable from \eqref{intVbeta 1d} through a suitable choice of $\beta$. For instance, by setting $\beta(t)=t$ and choosing $F(x) \equiv 1$ or $F(x)=x$, identity \eqref{intVbeta 1d} yields, respectively:
\[ \mathbb{E}_{\mathbb{P}}[W_t] = 0, \quad \mathbb{E}_{\mathbb{P}}[W^2_t] = t. \]
Alternatively, if we choose a constant $\beta(t) \equiv 1$ in \eqref{intVbeta 1d}, we find that for any sufficiently smooth $G = F'$:
\[ \mathbb{E}_{\mathbb{P}}[G(W_t)] = \mathbb{E}_{\mathbb{P}}[G(W_0)], \]
reflecting the constancy of the expectation of Brownian motion, a property deeply tied to its martingale nature. Furthermore, by selecting the indicator function $\beta(u) = min(u,s)$ for a fixed $s < t$ and choosing $ F(x)=x$, equation \eqref{intVbeta 1d} specializes into the well-known covariance property (see \cite{baldi}):
\[ \mathbb{E}_{\mathbb{P}}[W_t W_s] = \min(t, s). \]
These examples demonstrate that our integration by parts theorem effectively encodes the fundamental invariance properties of the underlying diffusion.\\
\\
\noindent
Consider now a two-dimensional Brownian motion $d\mathbf{W}_t = (dW^1_t, dW^2_t)^T$ and the family of symmetries $V_\beta$ associated with rotational invariance, as derived in \eqref{Vbeta 2d}:
\[ V_{\beta} = \left(  \beta(t)y\partial_x - \beta(t)x\partial_y, \ 0, \  
\begin{pmatrix} 0 & \beta(t) \\ -\beta(t) & 0 \end{pmatrix},  
\begin{pmatrix} -y\beta'(t) \\ x\beta'(t) \end{pmatrix} \right). \]
Theorem \ref{teo integrazione per parti 2} ensures that for every arbitrary time-dependent function $\beta$ and for every functional $F \in C^2_b(\mathbb{R}^2)$, it holds:
\begin{equation}\label{ibp vb 2d}
\mathbb{E}_{\mathbb{P}} \left[ F(W^1_t, W^2_t) \int_0^t ( -W^2_s dW^1_s + W^1_s dW^2_s ) \beta'(s) \right] + \mathbb{E}_{\mathbb{P}} \left[ \beta(t) \left( W^2_t \partial_x F - W^1_t \partial_y F \right) \right] = 0.
\end{equation}
By choosing $\beta(t) = t$, we obtain:
\begin{equation}\label{levy area identity}
\mathbb{E}_{\mathbb{P}} \left[ F(W^1_t, W^2_t) \int_0^t ( W^1_s dW^2_s - W^2_s dW^1_s ) \right] = t \mathbb{E}_{\mathbb{P}} \left[ W^1_t \partial_y F(W^1_t, W^2_t) - W^2_t \partial_x F(W^1_t, W^2_t) \right].
\end{equation}
The stochastic integral on the left-hand side is precisely the {Lévy stochastic area}. This identity provides a characterization of the Lévy area through the lens of rotational symmetries; for instance, choosing $F \equiv 1$ immediately recovers its zero-mean property. \\ \noindent Alternatively, if we choose $\beta \equiv 1$, identity \eqref{ibp vb 2d} specializes into:
\[ \mathbb{E}_{\mathbb{P}}[W^2_t \partial_x F(W^1_t, W^2_t)] = \mathbb{E}_{\mathbb{P}}[W^1_t \partial_y F(W^1_t, W^2_t)], \]
which is a direct consequence of {Isserlis' theorem} applied to the Gaussian vector $(W^1_t, W^2_t)$.\\
\\
\noindent
For the sake of completeness, we examine the application of the theorem to the family $V_\alpha$, which is associated with time-invariance. Focusing on the one-dimensional case $dX_t = dW_t$, consider the symmetry vector field computed in \eqref{Valpha nd} :
\begin{equation*}
V_\alpha = \left( \frac{1}{2}\alpha'(t) x \partial_x, \ \alpha(t), \ 0,  \; - \frac{1}{2} \alpha''(t) x \right).
\end{equation*}
Theorem \ref{teo integrazione per parti 2} gives:
\begin{equation}\label{general_alpha}
-\alpha(t)\mathbb{E}_{\mathbb{P}}\left[\frac{1}{2}F''(W_t)\right] + \mathbb{E}_{\mathbb{P}}\left[F(W_t)\int_0^t -\frac{1}{2}\alpha''(s)W_s dW_s\right] + \mathbb{E}_{\mathbb{P}}\left[\frac{1}{2}\alpha'(t) W_t F'(W_t)\right] = 0.
\end{equation}
Choosing $\alpha(t) = t$ yields:
\begin{equation}\label{ibp va1}
t\mathbb{E}_{\mathbb{P}}[F''(W_t)] = \mathbb{E}_{\mathbb{P}}[ W_t F'(W_t)],
\end{equation}
while choosing $\alpha(t) = t^2$ yields:
\[ t^2\mathbb{E}_{\mathbb{P}}[F''(W_t)] = -2\mathbb{E}_{\mathbb{P}}\left[F(W_t)\int_0^t W_s dW_s\right] + 2t\mathbb{E}_{\mathbb{P}}[ W_t F'(W_t)]. \]
Substituting the stochastic integral $\int_0^t W_s dW_s = \frac{1}{2}(W_t^2 - t)$ and recalling \eqref{ibp va1}, we obtain:
\begin{equation}\label{ibp_second_order}
t^2\mathbb{E}_{\mathbb{P}}[F''(W_t)] = \mathbb{E}_{\mathbb{P}}[F(W_t)(W_t^2 - t)].
\end{equation}
Both \eqref{ibp va1} and \eqref{ibp_second_order} represent classical identities. In particular, for $t=1$, we recover the standard Gaussian integration by parts formulas for $Z \sim \mathcal{N}(0,1)$: $\mathbb{E}[Z F'] = \mathbb{E}[F'']$ and $\mathbb{E}[F(Z^2 - 1)] = \mathbb{E}[F'']$.\\
\\
\noindent
In summary, the classical identities of Gaussian analysis emerge here as 
unified manifestations of the underlying symmetry groups of the Wiener process. 
The flexibility in choosing the temporal profiles $\beta(t)$ and $\alpha(t)$ 
suggests that the integration by parts framework effectively acts as a 
generating mechanism for an infinite family of conservation laws for the 
diffusion. While we have restricted our attention to Brownian motion as a 
prototypical case for the sake of simplicity, the methodology developed 
herein is readily applicable to a broader class of general diffusions.

\bibliographystyle{plain}
\bibliography{references}
\end{document}